\documentclass[11pt, oneside]{amsart}

\usepackage{amsmath,amssymb,verbatim,amsthm,bbm, mathrsfs,amscd,pb-diagram}
\usepackage{aliascnt} 
\usepackage{appendix}
\DeclareMathAlphabet{\mathpzc}{OT1}{pzc}{m}{it}
\usepackage[all]{xy}
\usepackage{multirow,longtable}
\usepackage{diagbox}
\usepackage{makecell}
\usepackage{float} 
\usepackage{hyperref}
\usepackage{cleveref}
\usepackage{marginnote}

\usepackage{framed}

\usepackage{csquotes}

\usepackage{tikz}

\usepackage{multicol}

\usepackage{lscape} 

\usepackage{framed}

\usepackage{cite}

\usepackage{todonotes}

\newcounter{todocounter}

\makeatletter
\providecommand\@dotsep{5}
\renewcommand{\listoftodos}[1][\@todonotes@todolistname]{%
	\@starttoc{tdo}{#1}}
\makeatother

\usepackage{float}
\restylefloat{table}

\crefname{table}{table}{tables}
\crefname{listing}{Program-code}{Program-codes}  
\Crefname{listing}{Program-code}{Program-codes}
\crefname{subsection}{subsection}{subsections}

\theoremstyle{plain}

\newtheorem{Thm}{Theorem}[section]
\newtheorem{Cor}[Thm]{Corollary}
\newtheorem{Prop}[Thm]{Proposition}
\newtheorem{Lem}[Thm]{Lemma}
\newtheorem{Que}{Question}
\theoremstyle{definition}
\newtheorem{Remark}[Thm]{Remark}

\numberwithin{equation}{section}
\newcommand{\Card}[1]{\left\vert #1\right\vert} 
\newcommand{\Hecke}{\mathcal{H}} 
\newcommand{\coset}[1]{\left[ #1 \right]}  
\newcommand{\gen}[1]{\left\langle #1 \right\rangle}  

\newcommand{\Ker}{\operatorname{Ker}}
\newcommand{\Image}{\operatorname{Im}}
\newcommand{\Ind}{i}

\newcommand{\Hom}{\operatorname{Hom}}

\newcommand{\SL}{\operatorname{SL}}
\newcommand{\GL}{\operatorname{GL}}
\newcommand{\G}{\mathbb{G}}

\newcommand{\C}{\mathbb C}

\newcommand{\F}{\mathbb{F}}

\newcommand{\Z}{\mathbb{Z}}
\newcommand{\R}{\mathbb{R}}
\newcommand{\N}{\mathbb{N}}

\newcommand{\M}{\mathcal{M}}

\newcommand{\bk}[1]{\left(#1\right)} 
\newcommand{\bm}{\begin{multline*}}
\newcommand{\tu}{\end  {multline*}}

\DeclareMathOperator{\unif}{\varpi} 
\renewcommand{\check}[1]{#1 ^{\vee}} 
\DeclareMathOperator{\Real}{Re} 
\newcommand{\piece}[1]{\left\{\begin{matrix} #1 \end{matrix}\right.} 
\newcommand{\set}[1]{\left\{ #1 \right\}} 
\newcommand{\mvert}{\mathrel{}\middle\vert\mathrel{}} 
\newcommand{\res}[1]{\Big\vert_{#1}}
\newcommand{\suml}{\sum\limits}

\newcommand{\rmod}{/}
\newcommand{\lmod}{\backslash}
\newcommand{\Stab}{\operatorname{Stab}}

\newcommand{\esevencharp}[7]{\renewcommand*{\arraystretch}{1} \begin{pmatrix}&& #2 && \\ #1 & #3 & #4 & #5 & #6 &#7 \end{pmatrix} }

\newcommand{\esevenchar}[7]{\renewcommand*{\arraystretch}{1} \begin{bmatrix}&& #2 && \\ #1 & #3 & #4 & #5 & #6 &#7 \end{bmatrix} }

\newcommand{\esevencharnontriv}[7]{\renewcommand*{\arraystretch}{1} \begin{bmatrix}&& #2 && \\ #1, & #3, & #4, & #5, & #6, &#7 \end{bmatrix} }

\newcommand{\dsixcharchar}[6]{\renewcommand*{\arraystretch}{1} \begin{bmatrix} &&& #6 & \\ #1 & #2 & #3 & #4 & #5 \end{bmatrix} }

\newcommand{\bfX}{\mathbf{X}}

\makeatletter
\newcommand*{\rom}[1]{\expandafter\@slowromancap\romannumeral #1@}
\makeatother

\renewcommand*{\arraystretch}{1.5}

\makeatletter
\def\imod#1{\allowbreak\mkern10mu({\operator@font mod}\,\,#1)}
\makeatother

\makeatletter
\renewcommand\section{\@startsection{section}{1}{\z@}%
	{-3.5ex \@plus -1ex \@minus-.2ex}%
	{2.3ex \@plus.2ex}%
	{\center\normalfont\large\bfseries}}
\makeatother

\makeatletter
\renewcommand\subsection{\@startsection{subsection}{2}{\z@}%
	{-3.5ex \@plus -1ex \@minus-.2ex}%
	{2.3ex \@plus.2ex}%
	{\normalfont\large\bfseries}}
\makeatother

\makeatletter
\renewcommand\subsubsection{\@startsection{subsubsection}{3}{\z@}%
	{-3.5ex \@plus -1ex \@minus-.2ex}%
	{2.3ex \@plus.2ex}%
	{\normalfont\large\bfseries}}
\makeatother

\makeatletter
\newtheorem*{rep@theorem}{\rep@title} \newcommand{\newreptheorem}[2]{%
	\newenvironment{rep#1}[1]{%
		\def\rep@title{\bf #2 \ref{##1} }%
		\begin{rep@theorem} }%
		{\end{rep@theorem} } }
\makeatother
\newreptheorem{theorem}{Theorem}
\newreptheorem{lemma}{Lemma}

\protected\def\ignorethis#1\endignorethis{}
\let\endignorethis\relax

\usepackage{colortbl}
\usepackage{zref-savepos}
\usepackage{pict2e}

\newcounter{NoTableEntry}
\renewcommand*{\theNoTableEntry}{NTE-\the\value{NoTableEntry}}

\makeatletter
\newcommand*{\notableentry}{%
	\kern-\tabcolsep
	\stepcounter{NoTableEntry}%
	\vadjust pre{\zsavepos{\theNoTableEntry t}}
	\vadjust{\zsavepos{\theNoTableEntry b}}
	\zsavepos{\theNoTableEntry l}
	\raisebox{%
		\dimexpr\zposy{\theNoTableEntry b}sp
		-\zposy{\theNoTableEntry l}sp\relax
	}[0pt][0pt]{%
		\setlength{\unitlength}{1pt}%
		\edef\w{%
			\strip@pt\dimexpr\zposx{\theNoTableEntry r}sp%
			-\zposx{\theNoTableEntry l}sp\relax
		}%
		\edef\h{%
			\strip@pt\dimexpr\zposy{\theNoTableEntry t}sp%
			-\zposy{\theNoTableEntry b}sp\relax
		}%
		\ifdim\w pt=0pt 
		\else
		\begin{picture}(0,0)%
		\edef\x{%
			\noexpand\put(0,0){\noexpand\line(\w,\h){\w}}%
			\noexpand\put(0,\h){\noexpand\line(\w,-\h){\w}}%
		}\x
		\end{picture}%
		\fi
	}%
	\hspace{0pt plus 1filll}%
	\zsavepos{\theNoTableEntry r}
	\kern-\tabcolsep
}

\makeatletter
\providecommand*{\cupdot}{%
	\mathbin{%
		\mathpalette\@cupdot{}%
	}%
}
\newcommand*{\@cupdot}[2]{%
	\ooalign{%
		$\m@th#1\cup$\cr
		\sbox0{$#1\cup$}%
		\dimen@=\ht0 %
		\sbox0{$\m@th#1\cdot$}%
		\advance\dimen@ by -\ht0 %
		\dimen@=.5\dimen@
		\hidewidth\raise\dimen@\box0\hidewidth
	}%
}

\providecommand*{\bigcupdot}{%
	\mathop{%
		\vphantom{\bigcup}%
		\mathpalette\@bigcupdot{}%
	}%
}
\newcommand*{\@bigcupdot}[2]{%
	\ooalign{%
		$\m@th#1\bigcup$\cr
		\sbox0{$#1\bigcup$}%
		\dimen@=\ht0 %
		\advance\dimen@ by -\dp0 %
		\sbox0{\scalebox{2}{$\m@th#1\cdot$}}%
		\advance\dimen@ by -\ht0 %
		\dimen@=.5\dimen@
		\hidewidth\raise\dimen@\box0\hidewidth
	}%
}
\makeatother

\allowdisplaybreaks

\newcommand{\fun}[1]{\bar{\omega}_{#1}}
\newcommand{\jac}[3]{r^{#1}_{#2}\bk{#3}}
\newcommand{\weyl}[1]{\mathit{W}_{#1}}
\newcommand{\w}{w}
\newcommand{\mult}[2]{mult\bk{#1,#2}}
\newcommand{\s}[1]{s_{#1}}
\newcommand{\para}[1]{#1}
\newcommand{\inner}[1]{\langle #1 \rangle}
\newcommand{\ind}[3]{i^{#2}_{#1}\bk{#3}}
\newcommand{\Span}{\operatorname{Span}}

\newcolumntype{H}{>{\setbox0=\hbox\bgroup}c<{\egroup}@{}}
\title[Degenerate Principal Series of $E_7$]{The Degenerate Principal Series Representations of Exceptional Groups of Type $E_7$ over $p$-adic Fields - \today} 
\author[Hezi Halawi and Avner Segal]{Hezi Halawi${^{1}}$ and Avner Segal${^{2}}$}
\address{${^1}$ School of Mathematics, Ben Gurion University of the Negev, POB 653, Be'er Sheva 84105, Israel}
\address{${^2}$ Department of Mathematics,  Bar-Ilan University , Ramat-Gan, 5290002 Israel  }
\email{halawi@post.bgu.ac.il, segalavner@gmail.com}

\numberwithin{equation}{section}

\subjclass[2010]{22E50, 20G41, 20G05}
\usepackage[printwatermark]{xwatermark}
\usepackage{pifont}
\usepackage{slashbox}

\usepackage[showframe=false]{geometry}
\usepackage{changepage}

\usepackage{enumitem}
\usepackage{float}
\restylefloat{table}
\usepackage{xparse}
\NewDocumentCommand{\ceil}{s O{} m}{%
	\IfBooleanTF{#1} 
	{\left\lceil#3\right\rceil} 
	{#2\lceil#3#2\rceil} 
}
\newcommand{\divides}{\Big \vert}
\newcommand{\RR}{\makecell{$red.^*$}}
\newcommand{\RI}{\notableentry}
\newcommand{\NRR}{\makecell{$red.$}}
\newcommand{\NRI}{\makecell{$irr.$}}

\begin{document}
\maketitle
\begin{abstract}
In this paper, we study the degenerate principal series of a split, simply-connected, simple p-adic group of type $E_7$. We determine the points of reducibility and the maximal semi-simple subrepresentation at each point.
\end{abstract}
\tableofcontents

\section{Introduction}

This paper is a continuation of our ongoing project which studies the degenerate principal series of exceptional groups of type $E_n$.
 More precisely, let $F$ be a non-Archimedean local field and let  $G$ denote the group   of $F$-points of a  split, simple, simply-connected group of type $E_7$. We  answer  the following question:
\begin{Que}\label{main question}
Let $P$ be a standard proper maximal parabolic subgroup with a  Levi subgroup  $M$. Given a one dimensional representation   $\sigma$ of $M$, is the normalized parabolic induction $\pi =\Ind_{M}^{G}(\sigma)$ reducible? What is the length of its maximal semi-simple quotient and what is the length of its maximal semi-simple subrepresentation?
\end{Que}

We mention here that  \Cref{main question} was already studied for various groups such as:
\begin{itemize}
\item
For the  general linear group  it was answered in a wider generality  in
\cite{MR584084,MR0425030}, and the answer of \Cref{main question} for the special linear group follows from \cite{MR620252} \cite{MR1141803}.
\item Symplectic groups in \cite{MR1134591}.
\item Orthogonal groups  in \cite{MR2017065,MR1346929}.
	\item Exceptional group of type $G_2$ in \cite{MR1480543}.
	\item Exceptional group of type $F_4$ in \cite{MR2778237}.
	\item Exceptional group of type $E_6$ in \cite{E6}.
\end{itemize}  

The main reason that such a study has not been preformed for split exceptional groups of type $E_n$ before is the size and complex structure of their Weyl groups. The computations  required for such a study cannot be preformed manually in a  reasonable amount of time. To overcome this problem, we harnessed a computer for this task. As in our previous paper, the calculation is implemented  using \emph{Sagemath} \cite{sagemath}.

Understanding the structure of degenerate principal series is important for the studying automorphic representations.
For example, it is conjectured that in the right complex half-plane, a degenerate Eisenstein series would has a pole if and only if the local degenerate principal series is reducible for almost all primes. Moreover, the residual representation at such a point is a sum of restricted tensor products of quotients of the local degenerate principal series.

This paper is arranged as follows:
\begin{itemize}
\item
\Cref{section :: notations} introduces the notations used in this paper.  
\item
\Cref{section:: algo}  outlines our method. 
\item
\Cref{section::results} introduces the group G and its structure, states our main theorem, \Cref{theroem::main}, and its proof.
\item
\Cref{App:knowndata} contains the information on representations of Levi subgroups, required for the proof of \Cref{theroem::main}.
\item
In \Cref{section:: iwahori hecke algebra} we recall parts of the theory of Iwahori-Hecke algebras and use them to complete the proof of \Cref{theroem::main}.

\end{itemize}
\paragraph{\textbf{Acknowledgments}:}
This research was supported by the Israel Science Foundation, grant number 421/17 (Segal).

\section{Preliminaries and Notations}
\label{section :: notations}

\subsection{Group Structure}

Let $F$ be a non-Archimedean local field and  $G$ be the group of $F$-points of an arbitrary split reductive group defined over $F$. 

Fix a maximal split torus $T$ and assume that $rank(G)=n$. We denote by $\Phi_{G}$ the set of roots of $G$ with respect to $T$. Fixing a Borel subgroup $\para{B} \supset T$ determines a partition of $\Phi_{G}$ into positive roots, $\Phi^{+}_{G}$,  and negative roots, $\Phi^{-}_{G}$. Let $\Delta_{G}$ denote the set of simple roots of $G$ relative to $\para{B}$. 

For $\alpha \in \Phi_{G}$,  $\check{\alpha}$ stands for the associated co-root ,  $\fun{\alpha}$ stands for the associated fundamental weight, and $\check{\fun{\alpha}}$ stands for the associated co-fundamental weight, namely, 
\[
\inner{\beta,\check{\fun{\alpha}}}=
\inner{\fun{\beta},\check{\alpha}} = \delta_{\alpha,\beta}  \quad \forall \alpha,\beta \in \Delta_{G},
\]
where $\delta_{\alpha,\beta}$ is the Kronecker delta function.

Set $W=\weyl{G}$ to be the Weyl group of $G$ with respect to $T$. It is known that 
\[W= \inner{\s{\alpha} \: : \: \alpha \in \Delta_{G}},\]
where $\s{\alpha}$ is a simple reflection associated to $\alpha \in \Delta_{G}$.

For  $\Theta \subset \Delta_{G}$ we let $\para{P}_{\Theta}$
denote the parabolic subgroup of $G$ given by 
\[\para{P}_{\Theta} = \inner{\para{B} , \s{\alpha} \: : \: \alpha \in \Theta}.\]
Such subgroups are called  standard parabolic subgroups. Obviously, $\para{B} \subseteq \para{P}_{\Theta}$.
Each standard parabolic subgroup admits a Levi decomposition  $\para{P}_{\Theta}= M_{\Theta} N_{\Theta}$, where the factor $M_{\Theta}$ denotes the Levi subgroup of $\para{P}_{\Theta}$, and $N_{\Theta}$ denotes its unipotent radical.
In this case, $\Delta_{M_{\Theta}} =\Theta$ ,   $\Phi_{M_{\Theta}}  =  \bk{\Span_{\Z}{\Theta}} \cap \Phi_{G}$ and $\Phi_{M_{\Theta}}^{+}  =  \bk{\Span_{\Z}{\Theta}} \cap \Phi_{G}^{+}$. The Weyl group of $M_{\Theta}$ is given by  $\weyl{M_{\Theta}} =  \inner{s_{\alpha} \: : \:  \alpha \in \Theta}$. 
In particular, we let $U$ denote the unipotent radical $N_{\emptyset}$ of $\para{B}$.

Given an enumeration of $\Delta_{G}= \{\alpha_1,\dots, \alpha_n\}$ we fix a notation for (proper) maximal parabolic and Levi subgroups.
For $1 \leq i \leq n$ we let 
$\para{P}_{i}= \para{P}_{\Delta_{G} \setminus \set{\alpha_i}}$ and $M_{i} =  M_{\Delta_{G} \setminus \set{\alpha_i}}$.

\subsection{Characters, W-action and Stabilizers}

Let $M$ be a Levi subgroup of $G$. The complex manifold of (quasi) characters of $M$ is denoted by $\bfX(M)$ and its structure is described in \cite[Section 2]{E6}.     

We say that  $\chi \in \bfX(M)$ is of \textit{finite order} if there exists $k  \in \N$ such that  $\chi^{k} =1$;  the smallest such  $k$ is called the order of $\chi$ and is denoted by $ord(\chi)=k$.

By \cite[Section 2]{E6},
every $\Omega \in \bfX(M)$ has the form
\begin{equation}\label{formula :: char_levi}
\Omega =\Omega_{G}+ \sum_{\alpha  \in \Delta_{G} \setminus  \Delta_{M}} \Omega_{\alpha} \circ \fun{\alpha},
\end{equation}
where $\Omega_{\alpha} \in \bfX(F^\times)$ and 
$\Omega_G$ is the pull-back of a character in $\bfX(G\rmod \coset{G,G})$.

Any $\Omega\in \bfX(F^\times)$ can be written by
\begin{equation} \label{formula :: char_f_times}
\Omega = s+\chi,
\end{equation}
where $\chi\in\bfX(F^\times)$ is of finite order and $s\in\C$ should be interpreted as the character $s(x)=|x|^s$.
We write $\Real(\Omega),\Image(\Omega)\in\bfX(F^\times)$ for the elements given by
\[
\Real(\Omega)(x) = |x|^{\Real(s)},\quad 
\Image(\Omega)(x) = \chi(x) |x|^{\Image(s)} .
\]
Similarly, for $\Omega \in \bfX(M)$ given as in \Cref{formula :: char_levi}, we write
\[
\begin{array}{l}
\Real(\Omega) =\Real(\Omega_{G})+ \sum_{\alpha  \in \Delta_{G} \setminus  \Delta_{M}} \Real(\Omega_{\alpha}) \circ \fun{\alpha}, \\
\Image(\Omega) =\Image(\Omega_{G})+ \sum_{\alpha  \in \Delta_{G} \setminus  \Delta_{M}} \Image(\Omega_{\alpha}) \circ \fun{\alpha}.
\end{array}
\]

\vspace{0.2cm}

In particular, if $G$ is simple, then, by \Cref{formula :: char_levi} and \Cref{formula :: char_f_times}, any complex character of $M_{i}$ is of the form
\begin{equation}
\Omega_{\para{M}_i,s,\chi}^{G}  =(s+ \chi) \circ \fun{\alpha_i},
\label{formula:: character_form}
\end{equation} 
with $s\in \C$ and $\chi\in\bfX(F^\times)$ is of finite order.
We write  $\Omega_{M_i,s} = \Omega_{M_i,s, Triv}^{G}$, where $Triv$ stands for the trivial character of $M_i$.  
When there is no risk of confusion, we omit the subscript $M_i$ and the superscript $G$ and write $\Omega_{s,\chi}$ for simplicity.

We note that, since $\weyl{G}$ acts on $T$, it also acts on $\bfX(T)$.
The set
\[
\set{\lambda\in \bfX(T) \mvert \Real(\inner{\lambda,\check{\alpha}})\leq 0 \quad \forall \alpha  \in \Delta_{G}}
\]
is a fundamental domain in $\bfX(T)$ for the action of $W$.
An element $\lambda \in \bfX(T)$ is called \textbf{anti-dominant} if  $\Real \bk{\inner{\lambda,\check{\alpha}}} \leq 0 $ for every $\alpha  \in \Delta_{G}$.  
For every $\lambda \in \bfX(T)$, the Weyl-orbit of $\lambda$ contains an anti-dominant element (possibly more than one).
We also note that the orbit $\weyl{G}\cdot\lambda$ is finite.

By definition, for every $\lambda \in \bfX(T)$, one has
\[\Stab_{\weyl{G}}(\lambda) = \Stab_{\weyl{G}}\bk{\Real(\lambda)} \cap \Stab_{\weyl{G}}\bk{\Image(\lambda)}.\]

We note that the stabilizers of the elements in the orbit $\weyl{G} \cdot 
\lambda$ are conjugate, and hence it is enough to determine one of them in order to determine all of them.
For computational reasons, it is easier to calculate the stabilizer of an anti-dominant element $\lambda_{a.d} \in \weyl{G}\cdot\lambda$.

Choose an anti-dominant element $\lambda_{a.d} \in \bfX(T)$ in the $\weyl{G}$-orbit of $\lambda$.
In that case, one has 
\[\Stab_{\weyl{G}}(\Real(\lambda_{a.d})) = \inner{\s{\alpha} \; : \:  \inner{\Real(\lambda_{a.d}), \check {\alpha}}=0, \quad \alpha \in \Delta_{G}}\]
and
\begin{equation}
\label{formula :: stabilizer}
\Stab_{\weyl{G}}(\lambda_{a.d}) = \set{w\in \Stab_{\weyl{G}}(\Real(\lambda_{a.d}))  \mvert w\cdot \Image(\lambda_{a.d}) = \Image(\lambda_{a.d})} .
\end{equation}

\subsection{Representations}

Let  $\operatorname{Rep}(G)$ denote the category of admissible representations of $G$.
We denote by \linebreak$\Ind_{M}^{G} :  \operatorname{Rep}(M) \rightarrow \operatorname{Rep}(G)$  the functor of normalized induction from $M$ to $G$, and by $r^{G}_{M} : \operatorname{Rep}(G) \rightarrow \operatorname{Rep}(M)$ the functor of  normalized Jacquet functor.
The Jacquet functor $r_{M}^{G}$ is left-adjoint to $\Ind_{M}^{G}$,
that is, Frobenius reciprocity holds
\begin{equation} 
\Hom_G\bk{\pi,\ind{M}{G}{\tau}} \cong \Hom_M\bk{\jac{G}{M}{\pi},\tau} . \label{formula::Frobenius}
\end{equation}

In parts of this work, it is convenient to consider representations of finite length of a group $H$ as elements in the Grothendieck ring $\mathfrak{R}\bk{H}$ of $H$. Given a representation $\pi$ of $H$, we write $[\pi]$ for its class in $\mathfrak{R}\bk{H}$.
In particular, we write $[\pi]=[\pi_1]+[\pi_2]$ if for any irreducible representation $\sigma$ of $H$ one has
\[
mult\bk{\sigma,\pi} = mult\bk{\sigma,\pi_1} + mult\bk{\sigma,\pi_2}.
\]
Here, $mult\bk{\sigma,\pi}$ denotes the multiplicity of $\sigma$ in the Jordan-H\"older series of $\pi$.
Furthermore, we write $\pi\leq \pi'$ if, for any irreducible representation $\sigma$ of $H$, 
\[
\mult{\sigma}{\pi} \leq \mult{\sigma}{\pi'}.
\]

We quote  \cite[Lemma. 2.12]{MR0579172} , \cite[Theorem 6.3.6]{Casselman} which gives another property of $r^{G}_{M}$ and $\Ind^{G}_{M}$.
\begin{Lem}[Geometric Lemma]\label{Lemma::geomtric_lema} 
	For Levi subgroups $L$ and $M$ of $G$, let
	\[
	W^{M,L} = \set{w\in W\mvert w\bk{\Phi^+_M}\subset\Phi^+,\ w^{-1}\bk{\Phi^+_L}\subset\Phi^{+}}
	\]
	be the set of shortest representatives in $W$ of $W_L\lmod W\rmod W_M$.
	For a smooth representation $\Omega$ of $M$, the representation $r_L^G \Ind_M^G \Omega$, as an element of $\mathfrak{R}\bk{L}$, is given by:
	\begin{equation}
	\coset{r_L^G \Ind_M^G \Omega} = \sum_{w\in W^{M,L}} \coset{\Ind_ {L'}^{L}{ w \circ \jac{M}{M^{'}}{\Omega}}} ,
	\end{equation}
	where, for $w\in W^{M,L}$, 
	\[
	\begin{split}
	& M'=M\cap w^{-1}Lw  \\
	& L'=wM w^{-1}\cap L .
	\end{split}\]
\end{Lem}

We note that, since the Jacquet functor takes finite length representations to finite length representations, then for any Levi subgroup $M$ of a maximal parabolic subgroup and $\Omega\in\bfX(M)$, the $T$-module $\jac{G}{T}{\Ind_{M}^{G}(\Omega)}$,  considered as an element of $\mathfrak{R}(T)$, is a finite sum of one-dimensional representations of $T$. Each such representation of $T$ is called an \textbf{exponent} of $\Ind_{M}^{G}(\Omega)$.
Moreover,  
\begin{eqnarray}
\jac{M}{T}{\Omega} = \Omega\res{T}  - \rho_{M} ,  \label{formula::leadingexp}
\end{eqnarray}
where $\rho_{M} = |\cdot | \circ \bk{\frac{1}{2}\sum_{\gamma \in \Phi^{+}_{M}} \gamma}$.
The exponent $\lambda_0 =\jac{M}{T}{\Omega}$ is called the \textbf{leading exponent} of
$\Ind_{M}^{G}(\Omega)$.
We note that all exponents of $\Ind_{M}^{G}(\Omega)$ lie in the $\weyl{G}$-orbit of $\lambda_0$.

\subsection{Intertwining Operators} \label{subsec :: intertwining_operators}

Let $\bfX^{un}(T)$ denote the group of characters of $T$ of the form
\[
\lambda = \sum_{i=1}^n s_i\circ\fun{\alpha_i}, \quad \bk{s_1,...,s_n}\in\C^n .
\]
Given  $w \in \weyl{G}$ we let 
\[R(w) = \set{\alpha > 0 \: : \: w\alpha <0}.\]

We fix $\w \in \weyl{G}$.
For $\lambda\in \bfX^{un}(T)$ such that $\inner{\lambda,\check{\alpha}}>0$ for every $\alpha \in R(w)$, the integral
\begin{equation}\label{eq::intertwing}
\M_{w}(\lambda)f (g) =  \int_{U \cap w U w^{-1} \backslash U } f(\w^{-1}u g) du
\end{equation} 
converges for every $f\in \Ind_{T}^{G}(\lambda)$
and satisfies the following properties:

\begin{enumerate}[ref= (P\arabic{*}),label= (P\arabic{*})]
	\item \label{properties::intertwining::1}
	$\M_{w}(\lambda)$ admits a meromorphic continuation to all $\bfX^{un}(T)$ and defines an intertwining operator $M_w(\lambda):\Ind_{T}^{G}(\lambda)\to \Ind_{T}^{G}(w\cdot\lambda)$.
	\item \label{properties::intertwining::2}
	If $w =w_1w_2$ such that $l(w)=l(w_1)+l(w_2)$, then 
	$\M_{w}(\lambda) =  \M_{w_1}(w_2 \cdot \lambda) \circ \M_{w_2}(\lambda)$.
	\item \label{properties::intertwining::3}
	Suppose that $\inner{\lambda,\check{\alpha}}>0$ for some $\alpha \in \Delta_{G}$. Then, $\ker \M_{\s{\alpha}}(\lambda) \neq 0$ if and only if $\inner{\lambda,\check{\alpha}}=1$. 
\end{enumerate}

It is customary to use the normalized intertwining operator 
\[N_{w}(\lambda)=\prod_{\gamma \in R(w)} \frac{\zeta(\inner{\lambda,\check{\gamma}}+1)}{\zeta(\inner{\lambda,\check{\gamma}})}\M_{\w}(\chi),\]
where $\zeta(z) =  \frac{1}{1-q^{-z}}$ for $z \in \C \setminus\{0\}$.
The normalized intertwining operator $N_{w}(\lambda)$ satisfies the same properties as $\M_{w}(\lambda)$, while \ref{properties::intertwining::2} holds in an even wider generality, namely, 
\[N_{w_1w_2}(\lambda)= N_{w_1}(w_2 \cdot \lambda) \circ N_{w_2}(\lambda) \quad \forall \quad  w_1,w_2 \in \weyl{G}.\]

Set $z=\inner{\chi,\check{\alpha}}$ and assume that $\Real(z)>0$.
Then, by \cite[Section 6]{MR517138}, the operator $N_{\s{\alpha}}(\lambda)$ is holomorphic at $\lambda$.

\section{The Algorithm}\label{section:: algo}

In this section, we survey  the method used in this paper to determine the reducibility of degenerate principal series and their maximal semi-simple subrepresentation and quotient. These ideas go back to  works of Bernstein, Zelevenisky, Sally, Tadi\'c, Mu\'ic, Jantzen, Casselman, Iwahori and  others. For more information one should consult \cite[Section 3]{E6}. We fix a maximal parabolic subgroup $\para{P}= MN$ of a simple group $G$.
Let $\Omega = \Omega_{M,s,\chi}$ be as in \eqref{formula:: character_form} and $\pi = \Ind_{M}^{G}(\Omega)$.

We recall, from \cite[Subsection 3.1]{E6}, that if $\pi$ is reducible, then $|x|^{\Image(s)}$ is of finite order.
Hence, we assume that $s\in\R$.

We start by addressing the reducibility of $\pi$.
For this purpose we make a distinction between regular and non-regular cases.

The representation $\pi$ is called \textbf{regular}  if  $\Stab_{\weyl{G}}(\lambda) = \set{e}$ for any (and hence all) $\lambda \leq \jac{G}{T}{\pi}$.

We point out that the structure of $\pi$ in the regular case is completely determined by \cite[Theorem. 3.1.2]{MR1134591}, while in the non-regular case there currently is no such general result regarding its reducibility and structure.
In \Cref{subsection :: nonreg_red}, \Cref{subsection :: nonreg} and \Cref{subsection::length}, we outline tools which will allow us to solve \Cref{main question} in most cases.
The remaining cases are dealt with in \Cref{section :: result E7 }.

\subsection{Regular Cases}
We recall \cite[Theorem. 3.1.2]{MR1134591}.
The setting of this theorem is as follows.
Let $\pi=\Ind_{M}^{G}(\Omega)$ be a regular representation.
For any $\alpha_i \in \Delta_{G}$, let $L_{\alpha_i}$ be the Levi subgroup of the standard parabolic subgroup 
$Q_{\alpha_i} = \inner{\para{B},\s{\alpha_i}}$
and let
 $BZ_{i}(\pi)$ stand for the set of  representations 
$\Ind_ {L^{'}}^{L_{\alpha_i}}(w \circ \tau)$, where
$\w$ runs over all $\w \in W^{M,L_{\alpha_i}}$ and for each $\w \in W^{M,L_{\alpha_i}}$: 
\begin{enumerate}
\item
$L^{'}= wM w^{-1} \cap  L_{\alpha_i}$.
\item
$\tau$ is a component of $\jac{M}{M^{'}}{\Omega}$,
where $M^{'}= M \cap w^{-1}L_{\alpha_i}w$.
\end{enumerate}
 Under these assumptions one has:
\begin{Thm} \label{Thm::regularcase red}
The following are equivalent:
\begin{enumerate}[label=(REG\arabic*), ref = (REG\arabic*)]
\item\label{Thm::regularcase red::item1}
$\pi$ is irreducible.
\item\label{Thm::regularcase red::item2}
For every $i$ and for every  $ \sigma  \in BZ_{i}(\pi),$ $\sigma$ is irreducible.
\item\label{Thm::regularcase red::item3}
$\inner{\jac{M}{T}{\Omega},\check{\alpha}} \not \in \{\pm1\}$ for every $\alpha \in \Phi_{G}^{+} \setminus \Phi_{M}$.
\end{enumerate}
\end{Thm}
The equivalence of \ref{Thm::regularcase red::item1} and \ref{Thm::regularcase red::item2} is the content of \cite[Theorem. 3.1.2]{MR1134591}, while the equivalence of \ref{Thm::regularcase red::item2} and \ref{Thm::regularcase red::item3} is explained in \cite[Subsection 3.1]{E6}.

Thus, the regular reducible cases are determined by two linear conditions, \Cref{formula :: stabilizer} and \ref{Thm::regularcase red::item3}, and hence is simple to implement.
We conclude the discussion on the regular case by recalling from \cite[Subsection 3.1]{E6} that there are only finitely many non-regular $\Omega\in\bfX(M)$ and only finitely many regular $\Omega\in\bfX(M)$ such that $\Ind_M^G\Omega$ is reducible.

\vspace{0.4cm} 
In the remainder of this section, we assume that $\pi$ is non-regular.

\subsection{Non Regular Cases - Reducibility Test} \label{subsection :: nonreg_red}

To deal with the reducibility of non-regular representations, we quote the following reducibility criterion of Tadi\'c \cite[Lemma 3.1]{MR1658535},
\begin{Lem}[\textbf{RC}]\label{Tadic::irr} Let $\pi$ be a smooth representation of $G$ of finite-length.
Suppose that there is a Levi subgroup $L$ of $G$ and 
 there are smooth representations $\sigma$ and $\Pi$ of $G$ of finite-length such that 
\begin{enumerate}
\item
$\sigma\leq \Pi, \pi \leq \Pi$.
\item
$\jac{G}{L}{\pi} + \jac{G}{L}{\sigma} \not \leq \jac{G}{L}{\Pi}$. 
\item
$\jac{G}{L}{\pi} \not \leq \jac{G}{L}{\sigma}$.
\end{enumerate}
Then $\pi$ is reducible. Moreover, $\pi$ and $\sigma$ share a common irreducible subquotient.
\end{Lem}

In many cases we are able to prove the reducibility of $\pi = \Ind_{M}^{G}(\Omega)$
by applying \Cref{Tadic::irr} to the following data:
\begin{itemize}
	\item $\pi = \Ind_{M}^{G}(\Omega)$.
	\item $\Pi=\Ind_T^G(\lambda_{a.d.})$, where $\lambda_{a.d.}$ is an anti-dominant exponent in the $\weyl{G}$-orbit of $\jac{T}{M}{\Omega}$.
	\item $\sigma = \Ind_L^G\tau$, where $L$ is a standard Levi of $G$ and $\tau$ is a one-dimensional representation of $L$ such that $\jac{L}{T}{\tau}\in\weyl{G}\cdot\lambda_{a.d.}$.
\end{itemize}
Here condition \textit{(1)} is automatically satisfied.
By \cite[Lemma 3.4]{E6}, it holds that
\begin{equation} \label{formula:: full multiplicty}
\mult{\lambda_{a.d}}{\jac{G}{T}{\pi}} = \mult{\lambda_{a.d}}{\jac{G}{T}{\Ind_L^G\tau}} = |\Stab_{\weyl{G}}(\lambda_{a.d})|.
\end{equation}
Hence, condition \textit{(2)} of \Cref{Tadic::irr} is also satisfied.
Condition \textit{(3)}, on the other hand, needs to be checked for any candidate $\sigma = \Ind_L^G\tau$ by comparing $\mult{\lambda}{\jac{G}{T}{\pi}}$ and $\mult{\lambda}{\jac{G}{T}{\sigma}}$ for any $\lambda\in \jac{G}{T}{\pi}$.

We point out that, in most cases where $\pi$ is reducible, one can use a comparison with another degenerate principal series, that is, by taking $\sigma = \Ind_L^G\tau$ with $L$ being a maximal Levi subgroup too.

Note that $\tau$ is determined by $\jac{L}{T}{\tau}$.
It follows that, since the number of standard Levi subgroups is finite and so is $\weyl{G}\cdot\lambda_{a.d.}$, the number of possible candidates $\Ind_L^G\tau$ is finite.

%
%
%
%
%

\subsection{Non-Regular Case - Irreducibility Test}\label{subsection :: nonreg}

We now describe the main method of proving the irreducibility of $\pi =\ind{M}{G}{\Omega}$ implemented in this paper.

We consider the set of functions 
\[\mathcal{S} = \set{  f \: : \:  \bfX(T) \: \rightarrow \N \: : \: f \text{ has a finite  support } }.\]
Note that $\mathcal{S}$ has a natural partial order.
To any admissible representation $\sigma$ of $G$ we associate a function $f_{\sigma} \in \mathcal{S}$ by the following recipe
\[f_{\sigma}(\lambda) =  \mult{\lambda}{\jac{G}{T}{\sigma}}.\]

We say that a finite sequence $\set{f_{k}}_{k=0}^{n}$ in $\mathcal{S}$ is  \textbf{$\sigma$-dominated} if it satisfies:
\begin{enumerate}[label=$(F\arabic{*})$,]
\item\label{sequnce::1}
$f_{k} \leq f_{\sigma}$ for every $k\leq n$.
\item\label{sequnce::2}
$f_{k} \leq f_{k+1}$ for every $k\leq n$.
\setcounter{enumi}{2}
\item\label{sequnce::3}
There exists $\lambda^{'} \in \bfX(T)$ such that
$f_{0}(\lambda)=\begin{cases}
1 & \lambda=\lambda^{'} \\
0 & \lambda\neq\lambda^{'}.
\end{cases}
$
\end{enumerate} 
If $\set{f_{k}}_{k=0}^{n}$ satisfies only \ref{sequnce::1} and \ref{sequnce::2}, we say that it is \textbf{non-unital}.

\begin{Prop}
	\label{proposition :: branching_rules}
	If there exists an irreducible subquotient $\sigma$ of $\pi$ and a $\sigma$-dominated sequence $\set{f_{k}}_{k=0}^{n}$ in $\mathcal{S}$ such that $f_{n} = f_{\pi}$, then $\pi$ is irreducible and $\pi=\sigma$.
\end{Prop}

The proof for this can be found in \cite[Subsection 3.3]{E6}.
We use \Cref{proposition :: branching_rules} to show the irreducibility of $\pi$ in many cases by the following construction.

Fix an anti-dominant exponent $\lambda_{a.d}$ of $\pi$.
Since \[\mult{\lambda_{a.d}}{\jac{G}{T}{\pi}}=\mult{\lambda_{a.d}}{\jac{G}{T}{\Ind_{T}^{G}(\lambda_{a.d})}},\]
any irreducible representation $\sigma$ of $G$ such that $\lambda_{a.d}\leq \jac{G}{T}{\sigma}$ is a subquotient of $\pi$.
We fix such a $\sigma$ and construct a $\sigma$-dominated sequence $\set{f_{k}}_{k=0}^{n}$.
Note that $f_\sigma$ has finite support.

Let
\[
\Gamma = \set{(\lambda,L,\tau) \mvert \begin{matrix}
	\lambda\in r_T^G\pi \\
	\text{$L$ is a standard Levi subgroup of $G$} \\
	\text{$\tau$ is the unique irreducible representaion of $L$ such that $\lambda\leq\jac{L}{T}{\tau}$}
\end{matrix}} .
\]
The set $\Gamma$ is finite since $\weyl{G}\cdot\lambda_{a.d.}$ is finite, the number of standard Levi subgroups is finite and that for each $\lambda$ and $L$, $\Ind_T^L\lambda$ is of finite length.

We construct the sequence $\set{f_{k}}_{k=0}^{n}$ as follows:
\begin{enumerate}
	\item Let $f_0\in \mathcal{S}$ be defined by 
	$f_{0}(\lambda)= \delta_{\lambda,\lambda_{a.d}}$, where $\delta_{\lambda,\lambda_{a.d}}$ is the Kronecker delta function.
	
	\item Given the element $f_k$ and a triple $(\lambda,L ,\tau)\in\Gamma$, let $g\in\mathcal{S}$ be defined by
	\begin{equation}\label{formula::branchingrule}
	g(\mu) = 
	\max \set{f_{k}(\mu) , \ceil*[\big]{\frac{f_{k}(\lambda)}{\mult{\lambda}{\jac{L}{T}{\tau}}}} \cdot \mult{\mu}{\jac{L}{T}{\tau}}}. 
	\end{equation}
	
	\item If $g>f_k$, take $f_{k+1}=g$ and go back to step (2).
	\item If $g=f_k$ for all triples $(\lambda,L ,\tau)\in\Gamma$, take $n=k$.
	If $f_n=f_\pi$, then $\pi$ is irreducible.
\end{enumerate}

By \cite[Subsection 3.3]{E6}, the sequence $\set{f_{k}}_{k=0}^{n}$ a $\sigma$-dominated sequence.
Also, since $\Gamma$ is finite and there are only finitely many $g\in\mathcal{S}$ such that $g<f_\sigma$, this process terminates after a finite number of steps.
We also note, that $f_n$ is independent of the string of triples $(\lambda,L ,\tau)$ chosen in step (2).

Note that, if $f_{n} \neq f_{\pi}$, it does not imply that $\pi$ is reducible, as one can see from  \Cref{Prop::E7p50O1}.
Also, in this case, $f_n$ is $\sigma$-dominated for any subquotient $\sigma$ of $\pi$ which satisfy $\lambda_{a.d.}\leq \jac{G}{T}{\pi}$.

Throughout this paper, when we say a \textbf{branching rule calculation} we refer to a $\sigma$-dominated sequence $\set{f_{k}}_{k=0}^{n}$, not necessarily unital, constructed using steps (2) and (3) for a given $f_0$.
We then interpret the sequence as the following inference rule
\[
\suml_{\lambda\in \operatorname{Supp}{f_{0}}} f_0(\lambda)\times \coset{\lambda} \leq \jac{G}{T}{\sigma}
\quad \Rightarrow
\suml_{\lambda\in \operatorname{Supp}{f_{n}}} f_n(\lambda)\times \coset{\lambda} \leq \jac{G}{T}{\sigma}.
\]

The list of triples $(\lambda,L,\sigma)\in\Gamma$ which were used in this paper can be found in \Cref{App:knowndata}.
An explicit example for this process can be found in \cite[Appendix A]{E6} and the proof of \Cref{Prop::E7p50O1}.

\subsection{Length of Maximal Semi-Simple Subrepresentation}\label{subsection::length}

We now turn to the  calculation of the length of the maximal semi-simple subrepresentation and quotient of $\pi=\ind{M}{G}{\Omega}$ when $\pi$ is reducible. First we note that, by contragredience, it suffices to answer the former.
Recall that, if 
$\tau$ is a subrepresentation of $\pi= \Ind_{M}^{G}(\Omega)$, then, by induction in stages, one has
$\tau \hookrightarrow \Ind_{T}^{G}(\lambda_{0})$, where $\lambda_0 =\jac{M}{T}{\Omega}$. 
Thus, by \eqref{formula::Frobenius},
\begin{equation} \label{eq:: num_irr}
1\leq 
\dim_{\C} 
\Hom_T\bk{\jac{G}{T}{\tau},\lambda_0}.
\end{equation}
In particular, the length of the maximal semi-simple subrepresentation of $\pi$ is at most $\mult{\lambda_{0}}{\jac{G}{T}{\pi}}$.

Using \Cref{eq:: num_irr}, we phrase a criterion for $\pi$ to have a unique irreducible subrepresentation.  
\begin{Prop}\label{Prop::unique irr}
Let $\lambda_0 = \jac{M}{T}{\Omega}$ be the leading exponent of $\pi = \Ind_{M}^{G}(\Omega)$. 
\begin{enumerate}[ref =\Cref{Prop::unique irr}(\arabic*)] 
\item \label{Prop::unique::1}
For any irreducible subrepresentation $\tau$ of $\pi = \Ind_{M}^{G}(\Omega)$, it holds that $\lambda_0 \leq \jac{G}{T}{\tau}$.
\item \label{Prop::unique::2}
If $\tau$ is an irreducible constituent of $\pi$ which satisfy $\mult{\lambda_0}{\jac{G}{T}{\tau}} = \mult{\lambda_0}{\jac{G}{T}{\pi}}$, then $\pi$ admits a unique irreducible subrepresentation and this subrepresentation is $\tau$.
\end{enumerate}
\end{Prop}
\begin{proof}
The first item is a direct consequence of \Cref{formula::Frobenius}. For the second item, we argue as follows. 
By \ref{Prop::unique::1}, each irreducible subrepresentation $\tau$ of $\pi$ has $\mult{\lambda_0}{\jac{G}{T}{\tau}}\neq0$. However, there is exactly one irreducible constituent with that property. Hence, $\pi$ admits a unique irreducible subrepresentation.
\end{proof} 

As a result, in order to show that $\pi$ admits a unique irreducible subrepresentation it would be enough to show one of the following:
\begin{itemize}
\item
$\mult{\lambda_0}{\jac{G}{T}{\pi}} =1$.
\item
There exists a subquotient $\sigma$ of $\pi$ such that $\lambda_{a.d}\leq \jac{G}{T}{\sigma}$, and a $\sigma$-dominated sequence of functions $\set{f_i}_{i=0}^n$ such that $f_n(\lambda_0)=f_\pi(\lambda_0)$.
\end{itemize}

In some cases, where this approach does not suffice, we will use the following Corollary of \Cref{Prop::unique irr} to prove that $\pi$ admits a unique irreducible subrepresentation.
\begin{Cor} \label{cor::unique_irr}
Suppose $\pi= \Ind_{M}^{G}(\Omega)$ admits an embedding into $\Ind_{T}^{G}(\lambda)$ for some $\lambda\in \weyl{G}\cdot\lambda_0$, where $\lambda_0=\jac{M}{T}{\Omega}$, and there is a unique irreducible constituent $\tau$  of $\pi$ such that 
$\mult{\lambda}{\jac{G}{T}{\tau}} = \mult{\lambda}{\jac{G}{T}{\pi}}$. Then, $\pi$ has a unique irreducible subrepresentation.   
\end{Cor}

A direct computation shows that, if $\Real(s_0)>0$, then $\mult{\lambda_{0}}{\jac{G}{T}{\pi}}=1$. In particular, $\pi$ admits a unique irreducible subrepresentation. As a result we will consider only the cases where $\Real(s_0)\leq 0$.


\section{The Degenerate Principal Series of $E_7$}\label{section::results}

In this section we state and prove our main theorem using the algorithm and tools presented in \Cref{section:: algo}. The outline of this section is as follows:
\begin{itemize}
\item
In \Cref{section::structure}, we recall the structure of the split, simple, simply-connected exceptional group of type $E_7$.
\item
In \Cref{subsection:: main}, we state our main theorem about the points of reducibility of the degenerate principal series of $G$ and the length of the maximal semi-simple subrepresentation.
\item
In \Cref{subsection::proof} and \Cref{section :: result E7 } we prove the theorem. \Cref{subsection::proof} is dedicated to the cases where the algorithm in \Cref{section:: algo} yielded a complete answer, 
the remaining cases  are dealt with in \Cref{section :: result E7 }.
\end{itemize}

\subsection{The Structure of the Exceptional Group of Type $E_7$}\label{section::structure}

Let $G$ be the split, semi-simple, simply-connected group of type $E_7$. In this section we describe the structure of $G$ and set notations for the rest of the section. We fix a Borel subgroup $\para{B}$ and a maximal split torus $T\subset \para{B}$. 
The set of roots, $\Phi_{G}$, contains $126$ roots. The group $G$ is generated by symbols \[\set{x_{\alpha}(r) \: : \: \alpha \in \Phi_{G} ,r \in F}\]
subject to the Chevalley relations as in \cite[Section 6]{MR0466335}.

We label the simple roots $\Delta_{G}$ and the Dynkin diagram of $G$ as follows:
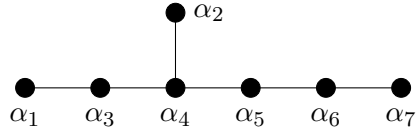
\[\begin{tikzpicture}[scale=0.5]
\draw (-1,0) node[anchor=east]{};
\draw (0 cm,0) -- (10 cm,0);
\draw (4 cm, 0 cm) -- +(0,2 cm);
\draw[fill=black] (0 cm, 0 cm) circle (.25cm) node[below=4pt]{$\alpha_1$};
\draw[fill=black] (2 cm, 0 cm) circle (.25cm) node[below=4pt]{$\alpha_3$};
\draw[fill=black] (4 cm, 0 cm) circle (.25cm) node[below=4pt]{$\alpha_4$};
\draw[fill=black] (6 cm, 0 cm) circle (.25cm) node[below=4pt]{$\alpha_5$};
\draw[fill=black] (8 cm, 0 cm) circle (.25cm) node[below=4pt]{$\alpha_6$};
\draw[fill=black] (10 cm, 0 cm) circle (.25cm) node[below=4pt]{$\alpha_7$};
\draw[fill=black] (4 cm, 2 cm) circle (.25cm) node[right=3pt]{$\alpha_2$};
\end{tikzpicture}\]

Recall that for  $\Theta \subset \Delta_{G}$ we denote by $M_{\Theta}$ the standard Levi subgroup  of $G$ such that $\Delta_{M} = \Theta$. If $\para{P}$ is a maximal parabolic subgroup, we let $M_{i}$ denote the Levi subgroup of $\para{P}_{i}= \para{P}_{\Delta_{G} \setminus \set{\alpha_i}}$. 

\begin{Lem}\label{Lemma::structrue::E7}
Under these notations, it holds that:
\begin{enumerate}[ref =\Cref{Lemma::structrue::E7}.(\arabic*)] 
\item \label{Lemma::structrue::E7 ::P1}
$M_{1} \simeq GSpin_{12}$.
\item \label{Lemma::structrue::E7 ::P2}
$M_{2} \simeq GL_{7}$.
\item \label{Lemma::structrue::E7 ::P3}
$M_{3}  =  \set{(g_1,g_2)\in GL_2 \times GL_6 \: : \:  \det g_1 = \det g_2}$.
\item \label{Lemma::structrue::E7 ::P4}
$M_{4}  =  \set{(g_1,g_2,g_3)\in GL_2 \times GL_3 \times GL_4 \: : \:  \det g_1 = \det g_2 =\det g_3}$.
\item \label{Lemma::structrue::E7 ::P5}
$M_{5}  =  \set{(g_1,g_2)\in GL_5 \times GL_3 \: : \:  \det g_1 = \det g_2}$.
\item \label{Lemma::structrue::E7 ::P6}
$M_{6}  =  \set{(g_1,g_2)\in GL_2 \times GSpin_{10} \: : \:  \det g_1 = \det g_2}$.
\item \label{Lemma::structrue::E7 ::P7}
$M_{7}\simeq GE_{6}$
\end{enumerate} 
\end{Lem}
We record here, for $1 \leq i \leq 7$, the cardinality of $W^{M_i,T}$, the set of shortest representatives of $\weyl{G} \slash \weyl{M_i}$
\begin{center}
\begin{tabular}{|c|c|c|c|c|c|c|c|}
\hline
 $i$ & 1 & 2 & 3 & 4 & 5 & 6 & 7 \\
 \hline 
$|W^{M_i,T}|$ & 126 & 576 & 2,016 & 10,080 & 4,032 & 756 & 56\\
\hline
\end{tabular} 
\end{center}
\vspace{0.5cm}
We also mention that $|\weyl{G}| =  2,903,040$.
By \Cref{formula :: char_levi},
every $\lambda \in \bfX(T)$ is of the form 
\[\lambda = \sum_{i=1}^{7} \Omega_{\alpha_{i}} \circ \fun{\alpha_i}.\]
As a shorthand, we will write
\[\esevencharp{\Omega_1}{\Omega_2}{\Omega_3}{\Omega_4}{\Omega_5}{\Omega_6}{\Omega_7}=\sum_{i=1}^{7} \Omega_{\alpha_{i}} \circ \fun{\alpha_i}.\]

Also, let 
$\esevenchar{\Omega_1}{\Omega_2}{\Omega_3}{\Omega_4}{\Omega_5}{\Omega_6}{\Omega_7}$ denote its class in $\mathfrak{R}(T)$.

\subsection{Main Theorem} \label{subsection:: main}
\begin{Thm}\label{theroem::main}
\begin{enumerate}
\item
For any $1 \leq  i \leq 7 $, all reducible regular $\Ind_{\para{M}_i}^{G}(\Omega_{\para{M}_i,s,\chi})$ and all non-regular $\Ind_{\para{M}_i}^{G}(\Omega_{\para{M}_i,s,\chi})$
are given in the following tables. For every maximal parabolic subgroup $\para{P}_i$, the table lists the values of $s\leq 0$  and $k=ord(\chi)$ such that $\pi$ is regular and reducible or non-regular. In particular, irr. stands for non-regular and reducible, red. stands for non-regular and reducible, while $red.^*$  stands for regular and reducible. For any triple $[i,s,k]$, not appearing in the table, the degenerate principal series $\Ind_{M_i}^{G}(\Omega_{M_i,s,\chi})$, with $ord(\chi)=k$, is regular and irreducible.
\begin{itemize} 
\item 
  For $\para{P}=\para{P}_{1}$ 
\begin{table}[H] 
 \caption{$\para{P}_{1}$-     Reducibility Points} 
 \begin{tabular}{|c|c|c|c|c|c|c|c|c|c|c|c|c|c|c|c|c|c|c|c} 
  \hline 
\diagbox{$ord\bk{\chi}$}{$s$}  & $-\frac{17}{2}$  & $-\frac{15}{2}$  & $-\frac{13}{2}$  & $-\frac{11}{2}$  & $-\frac{9}{2}$  & $-\frac{7}{2}$  & $-\frac{5}{2}$  & $-\frac{3}{2}$  & $-\frac{1}{2}$  & $0$  
 \\ \hline 
 $ 1 $ & \RR & \NRI & \NRI & \NRR & \NRI & \NRR & \NRI & \NRI & \NRR & \NRI  
 \\ \hline 
 $ 2 $ & \RI & \RI & \RI & \RI & \RI & \RI & \RI & \RI & \RR & \NRI  
 \\ \hline 
\end{tabular} 
  \label{Tab::E7::REG::P_1} 
 \end{table} 
\item 
  For $\para{P}=\para{P}_{2}$ 
\begin{table}[H] 
 \caption{$\para{P}_{2}$-     Reducibility Points} 
 \begin{tabular}{|c|c|c|c|c|c|c|c|c|c|c|c|c|c|c|c|c|c|c|c} 
  \hline 
\diagbox{$ord\bk{\chi}$}{$s$}  & $-7$  & $-6$  & $-5$  & $-4$  & $-3$  & $-2$  & $-\frac{3}{2}$  & $-1$  & $-\frac{1}{2}$  & $0$  
 \\ \hline 
 $ 1 $ & \RR & \NRI & \NRR & \NRR & \NRR & \NRR & \NRI & \NRR & \NRI & \NRI  
 \\ \hline 
 $ 2 $ & \RI & \RI & \RI & \RI & \RI & \NRR & \NRI & \NRI & \NRI & \NRR  
 \\ \hline 
\end{tabular} 
  \label{Tab::E7::REG::P_2} 
 \end{table} 
\item 
  For $\para{P}=\para{P}_{3}$ 
\begin{table}[H] 
 \caption{$\para{P}_{3}$-     Reducibility Points} 
 \begin{tabular}{|c|c|c|c|c|c|c|c|c|c|c|c|c|c|c|c|c|c|c|c} 
  \hline 
\diagbox{$ord\bk{\chi}$}{$s$}  & $-\frac{11}{2}$  & $-\frac{9}{2}$  & $-\frac{7}{2}$  & $-\frac{5}{2}$  & $-2$  & $-\frac{3}{2}$  & $-1$  & $-\frac{1}{2}$  & $-\frac{1}{6}$  & $0$  
 \\ \hline 
 $ 1 $ & \RR & \NRR & \NRR & \NRR & \NRI & \NRR & \NRI & \NRR & \NRI & \NRI  
 \\ \hline 
 $ 2 $ & \RI & \RI & \RI & \RR & \NRI & \NRR & \NRI & \NRR & \RI & \NRI  
 \\ \hline 
 $ 3 $ & \RI & \RI & \RI & \RI & \RI & \RI & \RI & \RR & \NRI & \RI  
 \\ \hline 
\end{tabular} 
  \label{Tab::E7::REG::P_3} 
 \end{table} 
\item 
  For $\para{P}=\para{P}_{4}$ 
\begin{table}[H] 
 \caption{$\para{P}_{4}$-     Reducibility Points} 
 \begin{tabular}{|c|c|c|c|c|c|c|c|c|c|c|c|c|c|c|c|c|c|c|c} 
  \hline 
\diagbox{$ord\bk{\chi}$}{$s$}  & $-4$  & $-3$  & $-2$  & $-\frac{3}{2}$  & $-1$  & $-\frac{2}{3}$  & $-\frac{1}{2}$  & $-\frac{1}{3}$  & $-\frac{1}{4}$  & $0$  
 \\ \hline 
 $ 1 $ & \RR & \NRR & \NRR & \NRR & \NRR & \NRR & \NRR & \NRI & \NRI & \NRI  
 \\ \hline 
 $ 2 $ & \RI & \RI & \RR & \NRR & \NRR & \RI & \NRR & \RI & \NRI & \NRR  
 \\ \hline 
 $ 3 $ & \RI & \RI & \RI & \RI & \RR & \NRR & \RI & \NRI & \RI & \NRI  
 \\ \hline 
 $ 4 $ & \RI & \RI & \RI & \RI & \RI & \RI & \NRR & \RI & \NRI & \NRI  
 \\ \hline 
\end{tabular} 
  \label{Tab::E7::REG::P_4} 
 \end{table} 
\item 
  For $\para{P}=\para{P}_{5}$ 
\begin{table}[H] 
 \caption{$\para{P}_{5}$-     Reducibility Points} 
 \begin{tabular}{|c|c|c|c|c|c|c|c|c|c|c|c|c|c|c|c|c|c|c|c} 
  \hline 
\diagbox{$ord\bk{\chi}$}{$s$}  & $-5$  & $-4$  & $-3$  & $-2$  & $-\frac{3}{2}$  & $-1$  & $-\frac{2}{3}$  & $-\frac{1}{2}$  & $-\frac{1}{3}$  & $0$  
 \\ \hline 
 $ 1 $ & \RR & \NRR & \NRR & \NRR & \NRR & \NRR & \NRI & \NRI & \NRI & \NRI  
 \\ \hline 
 $ 2 $ & \RI & \RI & \RI & \NRR & \NRR & \NRR & \RI & \NRI & \RI & \NRI  
 \\ \hline 
 $ 3 $ & \RI & \RI & \RI & \RI & \RI & \RR & \NRI & \RI & \NRI & \NRI  
 \\ \hline 
\end{tabular} 
  \label{Tab::E7::REG::P_5} 
 \end{table} 
\item 
  For $\para{P}=\para{P}_{6}$ 
\begin{table}[H] 
 \caption{$\para{P}_{6}$-     Reducibility Points} 
 \begin{tabular}{|c|c|c|c|c|c|c|c|c|c|c|c|c|c|c|c|c|c|c|c} 
  \hline 
\diagbox{$ord\bk{\chi}$}{$s$}  & $-\frac{13}{2}$  & $-\frac{11}{2}$  & $-\frac{9}{2}$  & $-\frac{7}{2}$  & $-\frac{5}{2}$  & $-2$  & $-\frac{3}{2}$  & $-1$  & $-\frac{1}{2}$  & $0$  
 \\ \hline 
 $ 1 $ & \RR & \NRR & \NRI & \NRR & \NRR & \NRI & \NRI & \NRI & \NRR & \NRI  
 \\ \hline 
 $ 2 $ & \RI & \RI & \RI & \RI & \RR & \NRI & \NRI & \NRI & \NRR & \NRI  
 \\ \hline 
\end{tabular} 
  \label{Tab::E7::REG::P_6} 
 \end{table} 
\item 
  For $\para{P}=\para{P}_{7}$ 
\begin{table}[H] 
 \caption{$\para{P}_{7}$-     Reducibility Points} 
 \begin{tabular}{|c|c|c|c|c|c|c|c|c|c|c|c|c|c|c|c|c|c|c|c} 
  \hline 
\diagbox{$ord\bk{\chi}$}{$s$}  & $-9$  & $-8$  & $-7$  & $-6$  & $-5$  & $-4$  & $-3$  & $-2$  & $-1$  & $0$  
 \\ \hline 
 $ 1 $ & \RR & \NRI & \NRI & \NRI & \NRR & \NRI & \NRI & \NRI & \NRR & \NRI  
 \\ \hline 
 $ 2 $ & \RI & \RI & \RI & \RI & \RI & \RI & \RI & \RI & \RI & \NRR
 \\ \hline 
\end{tabular} 
  \label{Tab::E7::REG::P_7} 
 \end{table} 
\end{itemize}  
\item
With the exception of 
the  cases listed below,
$\pi = \Ind_{M_{i}}^{G}(\Omega_{M_i,s,\chi})$
admits a unique irreducible subrepresentation. In the remaining cases, listed below, the length of the maximal semi-simple subrepresentation is $2$: 
\begin{enumerate}
\item
$\Ind_{M_{2}}^{G}(\Omega_{M_{2},-1,triv})$.
\item
$\Ind_{M_{2}}^{G}(\Omega_{M_{2},0,\chi})$,  where $ord(\chi)=2$.
\item
$\Ind_{M_{2}}^{G}(\Omega_{M_{2},-2,\chi})$,  where $ord(\chi)=2$.
\item
$\Ind_{M_{4}}^{G}(\Omega_{M_{4},0,\chi})$,  where $ord(\chi)=2$.
\item
$\Ind_{M_{4}}^{G}(\Omega_{M_{4},-\frac{1}{2},\chi})$,  where $ord(\chi)=4$.
\item
$\Ind_{M_{5}}^{G}(\Omega_{M_{5},-2,\chi})$,  where $ord(\chi)=2$.
\item
$\Ind_{M_{7}}^{G}(\Omega_{M_{7},0,\chi})$,  where $ord(\chi)=2$.
\end{enumerate}

\end{enumerate}
\end{Thm}

\begin{Remark}
	We point out that the results depend only on $ord(\chi)$ and not on the choice of $\chi$.
\end{Remark}

\begin{Remark}
	The reducibility of the degenerate principal series of $G$ can be partially studied using the results of \cite[Section 3]{MR1645956}, in the unramified case, and \cite{MR1993361}, in the case of $M=M_7$. The results of both agree with our calculations.
	
	Furthermore, by \cite{MR1993361}, the unique irreducible subrepresentation in the case $[7,-5,1]$ is the minimal representation of $G$. From the proof of Theorem 5.2, it follows that the minimal representation is also the unique irreducible subrepresentation of $[1,-11/2,1]$ and $[2,-5,1]$.
\end{Remark}

\subsection{Proof of 
\Cref{theroem::main} - Part I} \label{subsection::proof}

For most triples $(M_i,s,\chi)$, the proof of the reducibility or irreducibility of $\pi=\Ind_{M_i}^{G}(\Omega_{M_i,s,\chi})$ and the proof that it admits a unique irreducible subrepresentation can be performed using the algorithm outlined in \Cref{section:: algo}. More precisely:
\begin{itemize}
\item
For most proofs of reducibility, we use \Cref{Tadic::irr} with $\pi$ as above, $\Pi$ being $\Ind_{T}^{G}(\jac{M_i}{T}{\Omega_{M_1,s,\chi}})$ and $\sigma$ is given in Tables 8 through 15. A triple $[j,t,k]$ stands for a degenerate principal series $\sigma=\Ind_{M_j}^{G}(\Omega_{M_j,t,\chi})$, where $ord(\chi)=k$, while a triple $[(j_1,j_2),(s_1,s_2),(k_1,k_2)]$ stands for $\Ind_{M}^{G}(\Omega)$, where $M=M_{\Delta_G\setminus\{\alpha_{j_1},\alpha_{j_2}}\}$ and
\[ \Omega= (s_1 + \chi_1) \circ \fun{j_1} + (s_2 + \chi_2) \circ \fun{j_2}, \quad \text{where} \: \:  ord(\chi_i)=j_{i}.\]
\begin{itemize} 
\item 
  For $\para{P}=\para{P}_{1}$ 
\begin{table}[H] 
 \caption{Data for the proof of the reducibility of $\Ind_{M_  1}^{G}(\Omega_{M_{1,s,\chi}})$} 
 \begin{tabular}{|c|cH|} 
  \hline 
\diagbox{$s$}{$ord\bk{\chi}$}  & $1$  & $2$  
 \\ \hline 
 $ -\frac{11}{2} $ & $ \left[2, -5, 1\right] $&  \notableentry  
 \\ \hline 
 $ -\frac{7}{2} $ & $ \left[6, -\frac{7}{2}, 1\right] $&  \notableentry  
 \\ \hline 
 $ -\frac{1}{2} $ & $ \left[\left(2, 7\right), \left(1, 3\right), \left(0, 0\right)\right] $&  \notableentry  
 \\ \hline 
\end{tabular} 
  \label{Tab::E7::COMP::P_1} 
 \end{table} 
\item 
  For $\para{P}=\para{P}_{2}$ 
\begin{table}[H] 
 \caption{Data for the proof of the reducibility of $\Ind_{M_  2}^{G}(\Omega_{M_{2,s,\chi}})$} 
 \begin{tabular}{|c|c|c|} 
  \hline 
\diagbox{$s$}{$ord\bk{\chi}$}  & $1$  & $2$  
 \\ \hline 
 $ -5 $ & $ \left[1, -\frac{11}{2}, 1\right] $&  \notableentry  
 \\ \hline 
 $ -4 $ & $ \left[7, -2, 1\right] $&  \notableentry  
 \\ \hline 
 $ -3 $ & $ \left[1, -\frac{3}{2}, 1\right] $&  \notableentry  
 \\ \hline 
 $ -2 $ & $ \left[\left(6, 7\right), \left(2, 4\right), \left(0, 0\right)\right] $& $ \left[5, -2, 2\right] $ 
 \\ \hline 
 $ -1 $ & $ \left[3, -\frac{3}{2}, 1\right] $&  \notableentry  
 \\ \hline 
\end{tabular} 
  \label{Tab::E7::COMP::P_2} 
 \end{table} 
\item 
  For $\para{P}=\para{P}_{3}$ 
\begin{table}[H] 
 \caption{Data for the proof of the reducibility of $\Ind_{M_  3}^{G}(\Omega_{M_{3,s,\chi}})$} 
 \begin{tabular}{|c|c|cH|} 
  \hline 
\diagbox{$s$}{$ord\bk{\chi}$}  & $1$  & $2$  & $3$  
 \\ \hline 
 $ -\frac{9}{2} $ & $ \left[1, -\frac{13}{2}, 1\right] $&  \notableentry &  \notableentry  
 \\ \hline 
 $ -\frac{7}{2} $ & $ \left[7, -3, 1\right] $&  \notableentry &  \notableentry  
 \\ \hline 
 $ -\frac{5}{2} $ & $ \left[1, -\frac{1}{2}, 1\right] $&  \notableentry &  \notableentry  
 \\ \hline 
 $ -\frac{3}{2} $ & $ \left[2, -1, 1\right] $& $ \left[6, -\frac{1}{2}, 2\right] $&  \notableentry  
 \\ \hline 
 $ -\frac{1}{2} $ & $ \left[\left(2, 3\right), \left(1, 2\right), \left(0, 0\right)\right] $& $ \left[\left(2, 3\right), \left(1, 2\right), \left(0, 1\right)\right] $&  \notableentry  
 \\ \hline 
\end{tabular} 
  \label{Tab::E7::COMP::P_3} 
 \end{table} 
\item 
  For $\para{P}=\para{P}_{4}$ 
\begin{table}[H] 
 \caption{Data for the proof of the reducibility of $\Ind_{M_  4}^{G}(\Omega_{M_{4,s,\chi}})$ - part 1} 
 \begin{tabular}{|c|c|c|} 
  \hline 
\diagbox{$s$}{$ord\bk{\chi}$}  & $1$  & $2$   
 \\ \hline 
 $ -3 $ & $ \left[1, -\frac{11}{2}, 1\right] $&  \notableentry 
 \\ \hline 
 $ -2 $ & $ \left[1, -\frac{1}{2}, 1\right] $&  \notableentry 
 \\ \hline 
 $ -\frac{3}{2} $ & $ \left[6, -1, 1\right] $& $ \left[6, -1, 2\right] $  
 \\ \hline 
 $ -1 $ & $ \left[3, -\frac{1}{2}, 1\right] $& $ \left[3, -\frac{1}{2}, 2\right] $  
 \\ \hline 
 $ -\frac{2}{3} $ & $ \left[5, -\frac{1}{3}, 1\right] $&  \notableentry 
 \\ \hline 
 $ -\frac{1}{2} $ & $ \left[\left(2, 6\right), \left(3, \frac{5}{2}\right), \left(0, 0\right)\right] $& $ \left[\left(5, 7\right), \left(\frac{7}{2}, \frac{1}{2}\right), \left(1, 1\right)\right] $
 \\ \hline 
 $ 0 $ &  \notableentry & $ \left[\left(1, 4\right), \left(1, 2\right), \left(1, 0\right)\right] $ 
 \\ \hline 
\end{tabular} 
  \label{Tab::E7::COMP::P_4Part1} 
 \end{table}
 
 \begin{table}[H] 
  \caption{Data for the proof of the reducibility of $\Ind_{M_  4}^{G}(\Omega_{M_{4,s,\chi}})$ - part 2} 
  \begin{tabular}{|c|c|c|c|c|} 
   \hline 
 \diagbox{$s$}{$ord\bk{\chi}$}   & $3$  & $4$  
  \\ \hline 
  $ -\frac{2}{3} $  & $ \left[\left(4, 7\right), \left(\frac{10}{3}, -\frac{7}{3}\right), \left(1, 2\right)\right] $&  \notableentry  
  \\ \hline 
  $ -\frac{1}{2} $ &  \notableentry & $ \left[\left(2, 5\right), \left(\frac{1}{2}, \frac{5}{2}\right), \left(1, 1\right)\right] $ 
  \\ \hline 
 \end{tabular} 
   \label{Tab::E7::COMP::P_4Part2} 
  \end{table}
  
\item 
  For $\para{P}=\para{P}_{5}$ 
\begin{table}[H] 
 \caption{Data for the proof of the reducibility of $\Ind_{M_  5}^{G}(\Omega_{M_{5,s,\chi}})$} 
 \begin{tabular}{|c|c|cH|} 
  \hline 
\diagbox{$s$}{$ord\bk{\chi}$}  & $1$  & $2$  & $3$  
 \\ \hline 
 $ -4 $ & $ \left[7, -6, 1\right] $&  \notableentry &  \notableentry  
 \\ \hline 
 $ -3 $ & $ \left[1, -\frac{7}{2}, 1\right] $&  \notableentry &  \notableentry  
 \\ \hline 
 $ -2 $ & $ \left[2, -2, 1\right] $& $ \left[2, -2, 2\right] $&  \notableentry  
 \\ \hline 
 $ -\frac{3}{2} $ & $ \left[2, -\frac{1}{2}, 1\right] $& $ \left[2, -\frac{1}{2}, 2\right] $&  \notableentry  
 \\ \hline 
 $ -1 $ & $ \left[3, -\frac{1}{2}, 1\right] $& $ \left[\left(1, 2\right), \left(3, 3\right), \left(0, 1\right)\right] $&  \notableentry  
 \\ \hline 
\end{tabular} 
  \label{Tab::E7::COMP::P_5} 
 \end{table} 
\item 
  For $\para{P}=\para{P}_{6}$ 
\begin{table}[H] 
 \caption{Data for the proof of the reducibility of $\Ind_{M_  6}^{G}(\Omega_{M_{6,s,\chi}})$} 
 \begin{tabular}{|c|cH|} 
  \hline 
\diagbox{$s$}{$ord\bk{\chi}$}  & $1$  & $2$  
 \\ \hline 
 $ -\frac{11}{2} $ & $ \left[7, -7, 1\right] $&  \notableentry  
 \\ \hline 
 $ -\frac{7}{2} $ & $ \left[1, -\frac{7}{2}, 1\right] $&  \notableentry  
 \\ \hline 
 $ -\frac{5}{2} $ & $ \left[1, -\frac{1}{2}, 1\right] $&  \notableentry  
 \\ \hline 
 $ -\frac{1}{2} $ & $ \left[2, -1, 1\right] $& $ \left[3, -\frac{3}{2}, 2\right] $ 
 \\ \hline 
\end{tabular} 
  \label{Tab::E7::COMP::P_6} 
 \end{table} 
\item 
  For $\para{P}=\para{P}_{7}$ 
\begin{table}[H] 
 \caption{Data for the proof of the reducibility of $\Ind_{M_  7}^{G}(\Omega_{M_{7,s,\chi}})$} 
 \begin{tabular}{|c|cH|} 
  \hline 
\diagbox{$s$}{$ord\bk{\chi}$}  & $1$  & $2$  
 \\ \hline 
 $ -5 $ & $ \left[1, -\frac{11}{2}, 1\right] $&  \notableentry  
 \\ \hline 
 $ -1 $ & $ \left[1, -\frac{7}{2}, 1\right] $&  \notableentry  
 \\ \hline 
\end{tabular} 
  \label{Tab::E7::COMP::P_7} 
 \end{table} 
\end{itemize}

\item
The irreducibility of $\pi$ is proven, in most cases, using the algorithm in \Cref{subsection :: nonreg}. We start with an anti-dominant exponent $\lambda_{a.d}$ and a function $f_0\in \mathcal{S}$ given by $f_{0}= \delta_{\lambda_{a.d}}$ and an irreducible subquotient $\sigma$ of $\pi$ such that $\lambda_{a.d} \leq \jac{G}{T}{\sigma}$. Using the branching rules in \Cref{App:knowndata}, we  construct a $\sigma$-dominated sequence in $\mathcal{S}$ such that, at some point, $f_n=f_\pi$.
\item
For most reducible $\pi$ in the list which admits a unique irreducible subrepresentation, this can be proven using 
the algorithm described in \Cref{subsection::length}.

We point out that, in the case $[4, 0, 2]$, it follows from \Cref{Lemma::geomtric_lema} and \Cref{eq:: num_irr} that, since $\Ind_{M_  4}^{G}(\Omega_{M_{4,0,\chi}})$ is reducible and semi-simple such that $\jac{G}{T}{\Ind_{M_  4}^{G}(\Omega_{M_{4,0,\chi}})}$ contains the initial exponent with multiplicity $2$, then it is of length $2$.
\item
For a number of cases, listed below, the algorithm of \Cref{section:: algo}  was inconclusive. These cases are treated separately, with different methods, in \Cref{section :: result E7 }.
\begin{itemize}
\item
Irreducible points: $[4, 0, 1]$, $[5, 0, 1]$, $[5, 0, 2]$.
\item
Unique irreducible subrepresentation: $[4, -\frac{1}{2},4]$, $[4, -1, 1]$ and $[5, -1, 1]$.
\item
Maximal semi-simple subrepresentation of length $2$:
$[2, 0, 2]$, $[2, -1, 1]$, $[2, -2, 2]$ and $[5, -2, 2]$, $[7, 0, 2]$.
\end{itemize}
\end{itemize}

\subsection{Proof of \Cref{theroem::main} - Part II, the Remaining Cases} \label{section :: result E7 }
In this subsection we deal with the remaining cases in which the algorithm in \Cref{section:: algo} was inconclusive.
\begin{Prop}\label{Prop::E7p50O1}
The representation $\pi = \Ind_{M_5}^{G}(\Omega_{0})$ is irreducible.
\end{Prop} 
\begin{proof}
We begin by noting that $\pi$ is unitary and hence semi-simple. Therefore, in order to show irreducibility, it is enough to show that  it admits a unique irreducible subrepresentation. For this purpose, we fix the following notations 
\[
\begin{matrix}
\lambda_{a.d}= \esevenchar{0}{0}{0}{-1}{0}{0}{0}
& \lambda_0= \esevenchar{-1}{-1}{-1}{-1}{4}{-1}{-1}
\\
\lambda_{1} =  \esevenchar{0}{-1}{-1}{1}{-1}{0}{0}
& \lambda_{2} =\esevenchar{1}{-1}{0}{0}{-1}{0}{0} .
\end{matrix}
\]

and proceed as  follows:

\begin{itemize}
\item
Let $\sigma$ be an irreducible constituent  of $\pi$ having $\mult{\lambda_{a.d}}{\jac{G}{T}{\sigma}} \neq 0$. Using a sequence of branching rules we show that 
\[\mult{\lambda_1}{\jac{G}{T}{\sigma}} =  \mult{\lambda_1}{\jac{G}{T}{\pi}}=216.\]
This calculation is preformed in \Cref{tab::Prop::E7P5m0O1::1} below.
\item
Let $\tau$ be an irreducible constituent  of $\pi$ having $\lambda_{2}\leq \jac{G}{T}{\tau}.$
Applying a sequence of branching rules, detailed in \Cref{tab::Prop::E7P5m0O1::2} below, yields that 
\[\mult{\lambda_1}{\jac{G}{T}{\tau}}\geq 12.\]
It follows that $\tau =\sigma$ is the unique irreducible subquotient of $\pi$ having \[\mult{\lambda_2}{\jac{G}{T}{\tau}} \neq 0.\]
In particular, 
\[\mult{\lambda_2}{\jac{G}{T}{\sigma}} = \mult{\lambda_2}{\jac{G}{T}{\pi}}=72.\]

\item
Applying a sequence of branching rules, summarized in \Cref{tab::Prop::E7P5m0O1::3} below, it follows that
\[\mult{\lambda_0}{\jac{G}{T}{\sigma}} = \mult{\lambda_0}{\jac{G}{T}{\pi}}=2.\]
\end{itemize}
Therefore, by \ref{Prop::unique::2}, $\pi$ admits a unique irreducible subrepresentation, and hence, $\pi$ is irreducible.

We explain how the following tables should be read. Each line represents a branching rule of the type 
\[k \times \coset{\lambda} \leq \jac{G}{T}{\tau}  \Rightarrow l \times \coset{\mu} \leq \jac{G}{T}{\tau}.\]

The first two columns are $\lambda$ and $k$, then the third lists the type of the rule, as it is labeled in  \Cref{App:knowndata}, the forth is the Levi subgroup with respect to which it is applied. For shorthand,  we write $\set{b_1,\dots b_d}$ for the Levi subgroup $M$ which has $\Delta_{M} = \set{\alpha_j \: : \:  j \in \set{b_1,\dots b_d} } $. The fifth column lists the  Weyl element which is  applied. The last two columns list $\mu$ and $l$.

\begin{longtable}[H]{|c|c|c|c|c|c|c|c|c|c|} 
  \hline
  \multicolumn{10}{|c|}{$ \sigma $} \\ 
  \hline 
  { $\lambda$}  &
  {$k$}  &
  { Rule}  &
  {Levi} &
  { Weyl word  }  &
  {$\mu$}  &
  { $l$ }  \\
   \hline
   $\lambda_{a.d}$ & $1$ & OR  &  & & $\lambda_{a.d} = \esevenchar{0}{0}{0}{-1}{0}{0}{0}$ & $288$  \\  
   \hline
   $\lambda_{a.d}$ & $288$ & $A_4$ 
   &
    $\set{4,5,6,7}$
     & $\s{\alpha_4}$  & $\lambda_{1} =  \esevenchar{0}{-1}{-1}{1}{-1}{0}{0}$ & $216$  \\  
 \hline
\caption{\Cref{Prop::E7p50O1}, part 1}
\label{tab::Prop::E7P5m0O1::1}
\end{longtable}

\begin{longtable}[H]{|c|c|c|c|c|c|c|c|c|c|} 
  \hline
  \multicolumn{10}{|c|}{$ \tau $} \\ 
  \hline 
  { $\lambda$}  &
  {$k$}  &
  { Rule}  &
  {Levi} &
  { Weyl word  }  &
  {$\mu$}  &
  { $l$ }  \\
   \hline
   $\lambda_{2}$ & $1$ & OR  &  & & $\lambda_{2} = \esevenchar{1}{-1}{0}{0}{-1}{0}{0}$ & $36$  \\  
   \hline
   $\lambda_{2}$ & $36$ & $A_3$ 
   &
    $\set{1,3,4}$
     & $\s{\alpha_3}\s{\alpha_1}$  & $\lambda_{1} =  \esevenchar{0}{-1}{-1}{1}{-1}{0}{0}$ & $12$  \\  
 \hline
\caption{\Cref{Prop::E7p50O1}, part 2}
\label{tab::Prop::E7P5m0O1::2}
\end{longtable}

\begin{longtable}{|c|c|c|c|c|c|c|c|c|c|} 
  \hline
  \multicolumn{10}{|c|}{$ \sigma $} \\ 
  \hline 
  { $\lambda$}  &
  {$k$}  &
  { Rule}  &
  {Levi} &
  { Weyl word  }  &
  {$\mu$}  &
  { $l$ }  \\
\hline
 $\lambda_{2}$ &  $72$ & $A_3$ & $\set{2,3,4}$& $\s{\alpha_4}\s{\alpha_2}$  & $\lambda_{3} =\esevenchar{1}{0}{-1}{1}{-2}{0}{0}$ & $24$ \\
\hline
 $\lambda_{3}$ & $24$ & $A_1$  & & $\s{\alpha_7}\s{\alpha_6}\s{\alpha_5}$   & $\lambda_{4} =\esevenchar{1}{0}{-1}{-1}{0}{0}{2}$ & $24$ \\
\hline
 $\lambda_{4}$ &  $24$ & $A_3$ &$\set{4,5,6}$   & $\s{\alpha_6}\s{\alpha_4}$  & $\lambda_{5} =\esevenchar{1}{1}{-2}{0}{-1}{0}{2}$ & $8$ \\
\hline

 $\lambda_{5}$ &  $8$ &  $A_1$  & & $\s{\alpha_4}\s{\alpha_3}\s{\alpha_6}\s{\alpha_7}$  & $\lambda_{6} =\esevenchar{-1}{-1}{0}{2}{-1}{-2}{0}$ & $8$ \\
\hline

 $\lambda_{6}$ &  $8$ & $A_2$ & $\set{1,3}$  &  $\s{\alpha_1}$  & $\lambda_{7} =\esevenchar{1}{-1}{-1}{2}{-1}{-2}{0}$ & $4$ \\
\hline
 $\lambda_{7}$ & $4$ & $A_1$   &  & $\s{\alpha_6}\s{\alpha_5}\s{\alpha_7}\s{\alpha_6} $   & $\lambda_{8} =\esevenchar{1}{-1}{-1}{-1}{0}{3}{-1}$ & $4$ \\
\hline

 $\lambda_{8}$ & $4$ &  $A_2$  & $\set{4,5}$  & $\s{\alpha_4}$   & $\lambda_{9} =\esevenchar{1}{-2}{-2}{1}{-1}{3}{-1}$ & $2$ \\
\hline
 $\lambda_{9}$ & $2$ & $A_1$   & & $\s{\alpha_5}\s{\alpha_4}\s{\alpha_2}\s{\alpha_3} $   & $\lambda_{0} =\esevenchar{-1}{-1}{-1}{-}{4}{-1}{-1}$ & $2$ \\
\hline
\caption{\Cref{Prop::E7p50O1}, part 3}
\label{tab::Prop::E7P5m0O1::3}
\end{longtable}

\end{proof}     

\begin{Prop}
The following representations admit a unique irreducible subrepresentation:
\begin{enumerate}
\item
$\Ind_{M_4}^{G}(\Omega_{M_{4},-1})$.
\item
$\Ind_{\para{M}_5}^{G}(\Omega_{M_{5},-1})$.
\end{enumerate} 
\end{Prop}
\begin{proof}
We start by recalling more properties of the standard intertwining operators.
Let $M$ be a standard Levi subgroup of $G$. To
$w \in \weyl{M}$ we may associate an intertwining operator 
\[\M_{w}^{M} \: : \:  \Ind_{T}^{M}(\lambda) \rightarrow \Ind_{T}^{M}(w \cdot \lambda)\]
on the unramified principal series $\Ind_{T}^{M}(\lambda)$ in a similar fashion to  that of \Cref{subsec :: intertwining_operators}. By \Cref{eq::intertwing} and the induction in stages $\Ind_{T}^{G}(\lambda) =\Ind_{M}^{G}\bk{\Ind_{T}^{M}(\lambda)}$, it follows that:
\begin{equation}\label{eq ::levi_intertwining}
\M_{w}(f)(g) = \M_{w}^{M}(f(g)) \quad \forall f \in \Ind_{M}^{G}\bk{ \Ind_{T}^{M}(\lambda)}, \quad g \in G.
\end{equation}

Given an irreducible subrepresentation of $\sigma$ of $\Ind_{T}^{M}(\lambda)$ such that $\M_{w}^{M}\res{\sigma}$ is injective , it follows from \Cref{eq ::levi_intertwining} that $\M_{w}\res{\Ind_{M}^{G}(\sigma)}$ is also injective.    
\begin{enumerate}
\item
Denote $M = M_{\set{\alpha_2,\alpha_3,\alpha_5,\alpha_6}}$ and $L = M_{\set{\alpha_2,\alpha_3,\alpha_5,\alpha_4,\alpha_6}}$ and let $Triv_{M}$ be the trivial representation of $M$.
 By induction in stages, one has 
\[
\begin{split}
\Ind_{M_4}^{G}(\Omega_{M_{4},-1}) &\hookrightarrow \Ind_{M_4}^{G}  ( \Ind_{M}^{M_4}(Triv_{M})\otimes  \Omega_{M_{4},-1}) \\
&\hookrightarrow \Ind_{M}^{G}(-\fun{_{4}} \underbrace{-\frac{1}{2}( 3\fun{1}+ 6\fun{4}+4\fun{7})}_{\rho_M} +\frac{1}{2} \underbrace{(8\fun{4})}_{\rho_{M_4}}) \simeq \Ind_{M}^{G}(-\frac{3}{2}\fun{1}-2\fun{7})\\
&
\simeq 
\Ind_{L}^{G}( \Ind_{M}^{L}(\Omega_{M,0}^{L}) \otimes
-\frac{3}{2}\fun{1}-2\fun{7}). 
\end{split}
\]
Denote the right hand side by $\Pi$. 
Let $u= \s{\alpha_4 }\s{\alpha_2 }\s{\alpha_3 }\s{\alpha_5 }\s{\alpha_4 }$. By \Cref{D5:reseults} (after relabeling), the operator $\M_{u}(\lambda_{0})$ induces an injection,
\[
\Pi \hookrightarrow \Ind_{T}^{G}(\lambda_{a.d}),
\]
where $\lambda_{a.d}$ is the anti-dominant exponent of $\pi$ and $\lambda_{0}=\jac{M_4}{T}{\Omega_{M_4,-1}}$. By the Langlands' unique irreducible  subrepresentation theorem \cite[Section 1]{MR1665057}, $\Pi$ admits a unique irreducible subrepresentation. Since $\pi \hookrightarrow \Pi$, so does $\pi$.
\item
In order to show that $\pi =\Ind_{M_5}^{G}(\Omega_{M_{5},-1})$ admits a unique irreducible subrepresentation, we proceed as follows.
\begin{itemize}
\item
Let $\lambda_{a.d}$ be the anti-dominant exponent of $\pi$. Set $\sigma$ to be the unique irreducible subquotient of $\pi$ having $\mult{\lambda_{a.d}}{\jac{G}{T}{\sigma}} \neq 0$. In that case, it follows from the rule \eqref{Eq:OR}, see \Cref{App:knowndata}, that
\[\mult{\lambda_{a.d}}{\jac{G}{T}{\sigma}}= 72, \quad \text{ where } \quad \lambda_{a.d} =-\fun{4}-\fun{7}.\]
\item
On the other hand, by \Cref{Lemma::geomtric_lema}, it follows that $\mult{\lambda_{1}}{
\jac{G}{T}{\pi}}=48$, where $\lambda_1 =\s{\alpha_7}\lambda_{a.d}$.
\item
Applying the rule \eqref{Eq:: genan} on $\lambda_{a.d}$ with respect to $M_{\set{\alpha_5,\alpha_6,\alpha_7}}$ implies that
\[\mult{\lambda_1}{\jac{G}{T}{\sigma}} \geq 48=\mult{\lambda_1}{\jac{G}{T}{\pi}}.\]
\item
Hence, $\sigma$ is the unique irreducible constituent of $\pi$ with the property \[\mult{\lambda_1}{\jac{G}{T}{\sigma}} \neq 0.\]
\item
Let $M= M_{\set{\alpha_1,\alpha_2,\alpha_3,\alpha_4,\alpha_6}}$ and $ L = M_{\set{\alpha_1,\alpha_2,\alpha_3,\alpha_4, \alpha_5,\alpha_6}}$. Then, by induction in stages, one has
\[\begin{split}
\Ind_{M_5}^{G}(\Omega_{M_{5},-1}) &\hookrightarrow
\Ind_{M}^{G}\bk{-\fun{5} -\underbrace{\frac{1}{2}\bk{9\fun{5} +3\fun{7}}}_{\rho_M}
	+ \underbrace{5 \fun{5}  }_{\rho_{M_5}}
}
\\
\simeq & 
\Ind_{M}^{G} \bk{ -\frac{1}{2}\fun{5} -\frac{3}{2}\fun{7}} \simeq \Ind_{L}^{G}\bk{\Ind_{M}^{L}(\Omega_{M,-\frac{1}{2}}^{L} ) \otimes -\frac{7}{3}\fun{7}}.   \\
\end{split}\]

Denote the right hand side by 
$\Pi$.
\item
Let 
$u=\s{\alpha_4}
\s{\alpha_5}\s{\alpha_6}\s{\alpha_3}\s{\alpha_2}\s{\alpha_4}\s{\alpha_5}$. By \Cref{E6:reseults},  
$\M_{u}\res{\Pi}$ is injective. Hence, 
\[
\pi \hookrightarrow \Pi \hookrightarrow \Ind_{T}^{G}(\lambda_1).
\]
\item
Applying \Cref{cor::unique_irr}, the claim follows.
\end{itemize}
\end{enumerate}
\end{proof}

\begin{Prop}
\label{prop :: order_2_irr_Subs}
Let $\chi$ be a character of order 2. Then 
\begin{enumerate}
\item
$\Ind_{M_2}^{G}(\Omega_{M_2,0,\chi})$ and $\Ind_{M_{7}}^{G}(\Omega_{M_7,0,\chi})$ are semi-simple unitary representations of length $2$. 
\item $\Ind_{M_{2}}^{G}(\Omega_{M_2,-2,\chi})$ and $\Ind_{M_{5}}^{G}(\Omega_{M_5,-2,\chi})$ admit a maximal semi-simple representation of length $2$.
\item
$\Ind_{M_5}^{G}(\Omega_{M_5,0,\chi})$ is irreducible.
\end{enumerate}
\end{Prop}
\begin{proof}

In all cases, $\pi$ satisfies 
$\mult{\lambda_{0}}{\jac{G}{T}{\pi}}=2,$ 
where $\lambda_0$ is the leading exponent of $\pi$.

For each case, we will fix an anti-dominant exponent $\lambda_{a.d}$ and a Levi subgroup $L$ of $G$ such that 
$\Ind_{T}^{L}(\lambda_{a.d})$ is tempered and for each irreducible constituent $\sigma$ of $\Ind_{T}^{L}(\lambda_{a.d})$, $\Ind_{L}^{G}(\sigma)$ is a standard module, in the sense of \cite[Section 1]{MR2490651}.
We then use Langlands' unique irreducible subrepresentation theorem to determine the reducibility and number of irreducible subrepresentations of $\pi$.

In order to decompose $\Ind_{T}^{L}(\lambda_{a.d})$ into irreducible constituents, we first restrict it to the derived subgroup $L^{der}$ and decompose it as a representation of $L^{der}$.
We then study the way in which these irreducible representations of $L^{der}$ are glued into the irreducible constituents of $\Ind_{T}^{L}(\lambda_{a.d})$ as a representation of $L$.

In particular, if $\pi=\Ind_{M}^{G}(\Omega_{M,0,\chi})$, then $\pi$ is unitary and semi-simple of length at most $2$.
Hence, it is of length $2$ if and only if it is reducible.

\begin{enumerate}
\item
\underline{$\Ind_{M_2}^{G}(\Omega_{M_2,0,\chi})$}
Let $\pi=\Ind_{M_2}^{G}(\Omega_{M_2,0,\chi})$ and fix an anti-dominant exponent 
\[
\lambda_{a.d} = \esevencharnontriv{\chi}{\chi}{\chi}{-1}{\chi}{-1}{\chi}
\]
 of $\pi$.
We start by studying the principal series representation $i_T^L\lambda_{a.d.}$, where
$L= M_{\set{\alpha_1,\alpha_2,\alpha_3,\alpha_5,\alpha_7}}$.
We note that 
\[L^{der} = \SL_3^{\alpha_1,\alpha_3} \times
\SL_2^{\alpha_2} \times \SL_2^{\alpha_5} \times \SL_2^{\alpha_7},\]
where the superscripts indicate the simple roots in the copy of $SL_2$.
By \cite{MR620252}, it holds that
 \[\Ind_{T}^{L}(\lambda_{a.d})\res{L^{der}} =\bigoplus_{\epsilon_1,\epsilon_2,\epsilon_3 \in \{\pm 1\}}
 \underbrace{
  \sigma_{\chi}^{(3)} \boxtimes \sigma_{\epsilon_1,\chi}^{(2)}
 \boxtimes \sigma_{\epsilon_2,\chi}^{(2)}
 \boxtimes \sigma_{\epsilon_3,\chi}^{(2)}}_{\tau_{\epsilon_1,\epsilon_2,\epsilon_3}},\] where 
  $\sigma_{\chi}^{(3)} =\Ind_{T_{\SL_3}}^{\SL_3}(\chi \circ(\fun{1}+ \fun{3}))$ is irreducible, and $\Ind_{T_{\SL_2}}^{\SL_2}(\chi) = \sigma_{1,\chi}^{(2)} \oplus \sigma_{-1,\chi}^{(2)}$ is semi-simple of length $2$.
  
  Let $\varpi$ be the uniformizer. We recall that 
\[{\small \left(\begin{array}{rr}
  1 & 0 \\
  0 & \varpi
  \end{array}\right)} \cdot \sigma_{\epsilon,\chi}^{(2)} = \sigma_{-\epsilon,\chi}^{(2)}.\]

Hence, $\check{\alpha_4}(\varpi) \tau_{\epsilon_1,\epsilon_2,\epsilon_3} = \tau_{-\epsilon_1,-\epsilon_2,\epsilon_3}$ and 
$\check{\alpha_6}(\varpi) \tau_{\epsilon_1,\epsilon_2,\epsilon_3} = \tau_{\epsilon_1,-\epsilon_2,-\epsilon_3}.$

Since $L =\inner{L^{der},\check{\alpha_4}(x_1) , \check{\alpha_6}(x_2) \: : \:  x_1,x_2 \in F}$, it follows that 
\[\Ind_{T}^{L}(\lambda_{a.d}) = \tau_{-1} \oplus \tau_{1},\]
where
$\tau_{\epsilon} = \sum_{ \epsilon_1\epsilon_2\epsilon_3=\epsilon} \tau_{\epsilon_1,\epsilon_2, \epsilon_3}$ are irreducible tempered representations of $L$. Namely, $\Ind_{T}^{L}(\lambda_{a.d})$ is semi-simple of length $2$, and each irreducible constituent is glued out of $4$ irreducible representations of $L^{der}$.

By the Langlands' subrepresentation theorem, each $\Ind_{L}^{G}(\tau_{\epsilon})$ admits a unique irreducible subrepresentation, say $\Pi_{\epsilon}$. By \ref{Prop::unique::1},
\[\mult{\lambda_{a.d}}{\jac{G}{T}{\Pi_{\epsilon}}} \geq 1.\]

In particular, there are at least two irreducible subquotients of $\Ind_{T}^{G}(\lambda_{a.d})$ having $\lambda_{a.d}$ as an exponent. Since 
\[\mult{\lambda_{a.d}}{\jac{G}{T}{\pi}} = \mult{\lambda_{a.d}}{\jac{G}{T}{\Ind_{T}^{G}(\lambda_{a.d})}},\]
it follows that $\pi$ is reducible. 

\underline{$\Ind_{M_7}^{G}(\Omega_{M_7,0,\chi})$}
Let $\pi=\Ind_{M_7}^{G}(\Omega_{M_7,0,\chi})$ and fix an anti-dominant exponent
\[
\lambda_{a.d} = \esevencharnontriv{-1}{\chi}{-1}{-1}{\chi}{-1}{\chi}
\]
of $\pi$.
Similarly to the previous case, we study the principal series representation $i_T^L\lambda_{a.d.}$, where
$L= M_{\set{\alpha_2,\alpha_5,\alpha_7}}$.
We note that 
\[
L^{der} = \SL_2^{\alpha_2} \times
\SL_2^{\alpha_5} \times \SL_2^{\alpha_7}
\]
By \cite{MR620252}, it holds that
\[
\Ind_{T}^{L}(\lambda_{a.d})\res{L^{der}} =\bigoplus_{\epsilon_1,\epsilon_2,\epsilon_3 \in \{\pm 1\}}
\underbrace{
\sigma_{\epsilon_1,\chi}^{(2)}
\boxtimes \sigma_{\epsilon_2,\chi}^{(2)}
\boxtimes \sigma_{\epsilon_3,\chi}^{(2)}}_{\tau_{\epsilon_1,\epsilon_2,\epsilon_3}}.
\]
Since $L =\inner{L^{der},\check{\alpha_1}(x_1) , \check{\alpha_3}(x_2), \check{\alpha_4}(x_3) \check{\alpha_6}(x_4) \: : \:  x_1,x_2,x_3,x_4 \in F}$, it follows that 
\[\Ind_{T}^{L}(\lambda_{a.d}) = \tau_{-1} \oplus \tau_{1},\]
where
$\tau_{\epsilon} = \sum_{ \epsilon_1\epsilon_2\epsilon_3=\epsilon} \tau_{\epsilon_1,\epsilon_2, \epsilon_3}$ are irreducible tempered representations of $L$. Namely, $\Ind_{T}^{L}(\lambda_{a.d})$ is semi-simple of length $2$, and each irreducible constituent is glued out of $4$ irreducible representations of $L^{der}$.

The remainder of the argument is identical to the previous case.

\item
Let $\pi$ be $\Ind_{M_{2}}^{G}(\Omega_{M_2,-2,\chi})$ or $\Ind_{M_{5}}^{G}(\Omega_{M_5,-2,\chi})$.
Both representations admit the following anti-dominant exponent,
\[
\lambda_{a.d} = \esevencharnontriv{-1}{\chi}{\chi}{-1}{\chi}{-1}{\chi} .
\]
Let $L= M_{\set{\alpha_2,\alpha_3,\alpha_5,\alpha_7}}$ and note that 
\[
L^{der} = \SL_2^{\alpha_2} \times \SL_2^{\alpha_3} \times \SL_2^{\alpha_5} \times \SL_2^{\alpha_7} .
\]
It holds that
\[
\Ind_{T}^{L}(\lambda_{a.d})\res{L^{der}} =\bigoplus_{\epsilon_1,\epsilon_2,\epsilon_3,\epsilon_4 \in \{\pm 1\}}
\underbrace{
\sigma_{\epsilon_1,\chi}^{(2)}
\boxtimes \sigma_{\epsilon_2,\chi}^{(2)}
\boxtimes \sigma_{\epsilon_3,\chi}^{(2)}
\boxtimes \sigma_{\epsilon_4,\chi}^{(2)}}_{\tau_{\epsilon_1,\epsilon_2,\epsilon_3,\epsilon_4}}.
\]

Since $L =\inner{L^{der},\check{\alpha_1}(x_1) , \check{\alpha_4}(x_2) \check{\alpha_6}(x_3) \: : \:  x_1,x_2,x_3 \in F}$, it follows that 
\[\Ind_{T}^{L}(\lambda_{a.d}) = \tau_{-1} \oplus \tau_{1},\]
where
$\tau_{\epsilon} = \sum_{ \epsilon_2\epsilon_3\epsilon_4=\epsilon} \tau_{\epsilon_1,\epsilon_2, \epsilon_3,\epsilon_4}$ are irreducible tempered representations of $L$.
Namely, $\Ind_{T}^{L}(\lambda_{a.d})$ is semi-simple of length $2$, and each irreducible constituent is glued out of $8$ irreducible representations of $L^{der}$.

By the Langlands' subrepresentation theorem, each $\Ind_{L}^{G}(\tau_{\epsilon})$ admits a unique irreducible subrepresentation, say $\Pi_{\epsilon}$.
Furthermore, by \ref{Prop::unique::1}, it holds that
\[
\mult{\lambda_{a.d}}{\jac{G}{T}{\Pi_{\epsilon}}} \geq 1.
\]
Since 
\[
\mult{\lambda_{a.d}}{\jac{G}{T}{\pi}} = \mult{\lambda_{a.d}}{\jac{G}{T}{\Ind_{T}^{G}(\lambda_{a.d}}} = 2,
\]
it follows that $\Pi_1\oplus\Pi_{-1}$ is the maximal semi-simple subrepresentation of $\Ind_T^G\lambda_{a.d.}$.
On the other hand, 
\[
\mult{\lambda_{a.d}}{\jac{G}{T}{\pi}} = \mult{\lambda_{a.d}}{\jac{G}{T}{\Ind_{T}^{G}(\lambda_{a.d}}} = 2,
\]
and hence, both $\Pi_1$ and $\Pi_{-1}$ are subquotients of $\pi$.
On the other hand, a branching rule calculation shows that
\[
\mult{\lambda_{a.d}}{\jac{G}{T}{\Pi_\epsilon}} \geq 1 .
\]
Let $\lambda_0$ be the initial exponent of $\pi$ and
\[
w=\piece{
	\s{\alpha_2}\s{\alpha_3}\s{\alpha_5}\s{\alpha_4}\s{\alpha_6}\s{\alpha_5}\s{\alpha_1}\s{\alpha_3}\s{\alpha_4}\s{\alpha_2}, & \pi=\Ind_{M_{2}}^{G}(\Omega_{M_2,-2,\chi}) \\
	\s{\alpha_5}\s{\alpha_3}\s{\alpha_6}\s{\alpha_4}\s{\alpha_5}, & \pi=\Ind_{M_{5}}^{G}(\Omega_{M_5,-2,\chi}).
}
\]
Since $\M_{w}\bk{\lambda_0}$ is an isomorphism, it follows that $\pi\hookrightarrow \Ind_T^G\lambda_{a.d.}$.
It follows that $\Pi_1\oplus\Pi_{-1}$ is the unique irreducible subrepresentation of $\pi$.

\item
Let 
\[
\lambda_{a.d} = \esevencharnontriv{\chi}{\chi}{\chi}{-1}{\chi}{\chi}{\chi}
\]
and let $L =M_{4}$. We note that 
\[L^{der} =\SL_3^{\alpha_1,\alpha_3} \times
\SL_2^{\alpha_2} \times \SL_4^{\alpha_5,\alpha_6,\alpha_7}\]
and that 
\[\Ind_{T}^{L}(\lambda_{a.d})\res{L^{der}} = \bigoplus_{\epsilon \in \set{\pm1}} \underbrace{\sigma_{\chi}^{(3)} \boxtimes \sigma_{\epsilon,\chi}^{(2)} \boxtimes\sigma_{\chi}^{(4)}}_{\tau_{\epsilon}},\]
where $\sigma_{\epsilon,\chi}^{(2)},\sigma_{\chi}^{(3)}$ are as above and 
$\sigma_{\chi}^{(4)} = \Ind_{T_{\SL_4}}^{\SL_4}(\chi\circ(\fun{1} + \fun{2} + \fun{3}))$ is irreducible by \cite{MR620252}. 

Since $L = \inner{L^{der},\check{\alpha_4}(x) \: : \: x \in F}$
 and 
 $\check{\alpha_4}(\bar{\omega}) \cdot \tau_{\epsilon} = \tau_{-\epsilon}$, it follows that 
 \[
 \Ind_{T}^{L}(\lambda_{a.d}) = \tau,
 \]
 where $\tau =\tau_{1} \oplus \tau_{-1}$ is an irreducible tempered representation of $L$, glued from $2$ irreducible representations of $L^{der}$.
 Hence, $\Ind_{L}^{G}(\tau)$ is a standard module and, by the Langlands' subrepresentation theorem, it admits a unique irreducible subrepresentation $\Pi$. Furthermore, 
 $\mult{\Pi}{\Ind_{L}^{G}(\tau)} =1$.
 
 By \ref{Prop::unique::1}, $\Pi$ is the unique subquotient of $\Ind_{T}^{G}(\lambda_{a.d}) =\Ind_{M}^{G}(\tau)$ satisfying
 \[\mult{\lambda_{a.d}}{\jac{G}{T}{\Pi}}\neq 0.\]
 Otherwise, there would be a different subquotient $\Pi^{'}$of $\Ind_{T}^{G}(\lambda_{a.d})$ such that 
 \[\mult{\lambda_{a.d}}{\jac{G}{T}{\Pi^{'}}}\neq 0.\]
 This implies that either $\Pi \simeq \Pi^{'}$ or,  by a central character argument, see \cite[Lemma. 3.12]{E6}, that 
 \[\Pi^{'} \hookrightarrow \Ind_{T}^{G}(\lambda_{a.d}).\]
 Both would bring us to a contradiction.
 See \cite[\S 5]{MR2490651} for an alternative proof of this fact.
 
 We conclude  that 
 \[\mult{\lambda_{a.d}}{\jac{G}{T}{\Pi}} = \mult{\lambda_{a.d}}{\jac{G}{T}{\Ind_{T}^{G}(\lambda_{a.d}}} =16.\]
A branching rule calculation implies that
\[
16 \times \coset{\lambda_{a.d}} \leq \jac{G}{T}{\Pi} \Rightarrow 2 \times \coset{\lambda_{0}} \leq \jac{G}{T}{\Pi} .
\]
Namely, there exist a non-unital $\sigma$-dominated sequence $\set{f_{k}}_{k=0}^{n}$ in $\mathcal{S}$ such that $f_0(\lambda_{a.d.})=16$ and $f_n(\lambda_0)=2$.
Hence, $\Pi =\pi$ is irreducible.
\end{enumerate} 
\end{proof}

\begin{Remark}
	This method can also be used to prove that  $\Ind_{M_7}^{G}(\Omega_{M_7,0,\chi})$ is of length $2$.
\end{Remark}

Similarly, we prove the following.

\begin{Prop}\label{prop::lang_proof}
	Let $\chi$ be a finite character of order $4$.
	Then, the representation
	$\pi=\Ind_{M_4}^{G}(\Omega_{M_4,-\frac{1}{2},\chi})$ admits a unique irreducible subrepresentation.
\end{Prop}
\begin{proof}
	Let $\lambda_0=\jac{M_4}{T}{\Omega_{M_4,-\frac{1}{2},\chi}}$ and fix an anti-dominant exponent
	\[
	\lambda_{a.d} = \esevencharnontriv{-\frac{1}{2}+3\chi}{2\chi}{2\chi}{-\frac{1}{2}+\chi}{\chi}{-\frac{1}{2}+3\chi}{2\chi}
	\]
	of $\pi$.
	We remind the reader that we use an additive notation for $\bfX(F^\times)$, namely, 
	\[
	(n\chi)(x) = (\chi(x))^n .
	\]

	Let $L= M_{\set{\alpha_2,\alpha_3,\alpha_5,\alpha_7}}$ and note that 
	\[
	L^{der} = \SL_2^{\alpha_2} \times \SL_2^{\alpha_3} \times \SL_2^{\alpha_5} \times \SL_2^{\alpha_7} .
	\]
	We note that $2\chi$ is a character of order $2$.
	As in case (2) of \Cref{prop :: order_2_irr_Subs}, it holds that
	\[
	\Ind_{T}^{L}(\lambda_{a.d})\res{L^{der}} =\bigoplus_{\epsilon_1,\epsilon_2,\epsilon_3,\epsilon_4 \in \{\pm 1\}}
	\underbrace{
		\sigma_{\epsilon_1,2\chi}^{(2)}
		\boxtimes \sigma_{\epsilon_2,2\chi}^{(2)}
		\boxtimes \sigma_{\epsilon_3,2\chi}^{(2)}
		\boxtimes \sigma_{\epsilon_4,2\chi}^{(2)}}_{\tau_{\epsilon_1,\epsilon_2,\epsilon_3,\epsilon_4}}
	\]
	and 
	\[\Ind_{T}^{L}(\lambda_{a.d}) = \tau_{-1} \oplus \tau_{1},\]
	where
	$\tau_{\epsilon} = \sum_{ \epsilon_2\epsilon_3\epsilon_4=\epsilon} \tau_{\epsilon_1,\epsilon_2, \epsilon_3,\epsilon_4}$ are irreducible tempered representations of $L$.
	By the Langlands' subrepresentation theorem, each $\Ind_{L}^{G}(\tau_{\epsilon})$ admits a unique irreducible subrepresentation, say $\Pi_{\epsilon}$.
	Further more, a similar argument shows that $\Pi_1\oplus\Pi_{-1}$ is the maximal semi-simple subrepresentation of $\Ind_T^G\lambda_{a.d.}$ and that both $\Pi_1$ and $\Pi_{-1}$ are subquotients of $\pi$.
	
	The intertwining operator $\M_{w}\bk{\lambda_0}$, where 
	\[
	w= \s{\alpha_4} \s{\alpha_5} \s{\alpha_6} \s{\alpha_1} \s{\alpha_3} \s{\alpha_2} \s{\alpha_4} \s{\alpha_5} \s{\alpha_4} \s{\alpha_1} \s{\alpha_3} \s{\alpha_2} \s{\alpha_4},
	\]
	is an isomorphism and hence $\pi\hookrightarrow \Ind_T^G\lambda_{a.d.}$.
	It follows that $\Pi_1\oplus\Pi_{-1}$ is the maximal semi-simple subrepresentation of $\pi$.

	
%
\end{proof}

Next we turn to deal with two cases where we were not able to determine the length of the maximal semi-simple subrepresentation using the tools described in  \Cref{section:: algo}. 
We deal with these cases using a calculation in the Iwahori-Hecke algebra of $G$, which can be done since both of these representations are unramified.
We outline the proof here, while the part which uses the Iwahori-Hecke algebra is left to \Cref{section:: iwahori hecke algebra}.

\begin{Prop}\label{Prop::E7::iwahori}
\mbox{}

\begin{enumerate}
\item
The representation $\pi=\Ind_{M_{2}}^{G}(\Omega_{M_{2},-1})$ admits a  maximal semi-simple subrepresentation of length two. 
\item
The representation $\pi=\Ind_{M_{4}}^{G}(\Omega_{M_{4},0})$ is irreducible.
\end{enumerate}
\end{Prop}
\begin{proof}
	
Both cases are proven using the same approach.
We first outline our approach and then list the data required for each case, while postponing part of the calculation to \Cref{section:: iwahori hecke algebra}.
Let $\pi$ be one of the above representations, $\lambda_{0}$ (resp. $\lambda_{a.d}$) be the leading (resp. anti-dominant) exponent  of $\pi$ and note that $\mult{\lambda_{0}}{\jac{G}{T}{\pi}}=2$.
Hence, the maximal semi-simple subrepresentation of $\pi$ is of length at most 2.

We choose a subrepresentation $\sigma$ of $\pi$, a Weyl element $w\in W^{M,T}$ and an exponent $\lambda_1=w\cdot\lambda_0$, which satisfy the following properties:
\begin{itemize}
	\item $\mult{\lambda_1}{\jac{G}{T}{\pi}}= \mult{\lambda_1}{\jac{G}{T}{\sigma}}\neq 0$.
	In particular, $\sigma$ is the unique subquotient of $\pi$ with $\lambda_1$ as an exponent.
	\item $N_{w}(\lambda_0)\res{\pi}$ is non-zero.
\end{itemize}
Such triples $\bk{\sigma,w,\lambda_1}$ exist as will be shown bellow.
We start by explaining how such a triple can be used in order to determine the length of the maximal semi-simple subrepresentation of $\pi$.


Let $\bk{\sigma,w,\lambda_1}$ be such a triple.
It follows from \ref{Prop::unique::1} that $\sigma$ is the unique subquotient of $\pi$ which admits an embedding into $\Ind_T^G\lambda_1$.
It also follows that $\sigma$ appears in $\pi$ with multiplicity one.

By our assumptions, the image of $N_{w}(\lambda_0)\res{\pi}$ is a non-zero subrepresentation of $\Ind_T^G\lambda_1$ and hence,
\[
\lambda_1 \leq \jac{G}{T}{N_{w}(\lambda_0)(\pi)} .
\]
In particular, one concludes that $\sigma$ is the unique irreducible subrepresentation of the image $N_{w}(\lambda_0)(\pi)$.
Since it appears in $\pi$ with multiplicity one, it is not contained in the kernel.


On the other hand, any other irreducible subrepresentation of $\pi$ is necessarily contained in the kernel of $N_{w}(\lambda_0)\res{\pi}$.
More precisely,
\begin{itemize}
	\item
	Assume that $N_{w}(\lambda_0)\res{\pi}$ is not injective and let $\tau\neq 0$  be an irreducible subrepresentation of $\ker N_{w}(\lambda_0)\res{\pi}$.
	In particular, $\tau\neq\sigma$, since $\sigma$ appears in $\pi$ with multiplicity $1$.
	Hence, $\sigma\oplus\tau$ is the maximal semi-simple subrepresentation of $\pi$.
	
	\item
	Assume that $N_{w}(\lambda_0)\res{\pi}$ is injective.
	It follows that $\pi\hookrightarrow \Ind_T^G\lambda_1$.
	In this case, any subrepresentation of $\pi$, is a subrepresentation of $\Ind_T^G\lambda_1$.
	Since $\sigma$ is the unique subquotient of $\pi$ with that property, it follows that $\sigma$ is the unique irreducible subrepresentation of $\pi$.
\end{itemize}

\vspace{0.3cm}

It remains to show that there exists a choice of $(\sigma, w, \lambda_1)$ which satisfy the stated properties and determine the injectivity of $N_{w}(\lambda_0)\res{\pi}$.


First, let $\sigma$ be the unique (by \Cref{Lemma::unique anti}) subquotient of $\pi$ such that $\lambda_{a.d.}\leq \jac{G}{T}{\sigma}$.
In fact, by \Cref{Lemma::unique anti}, $\sigma$ appears in the principal series $\Ind_T^G\lambda_0$ with multiplicity one.
Fix a $\sigma$-dominated sequence $\set{f_{k}}_{k=0}^{n}$ in $\mathcal{S}$ constructed by the process described in \Cref{subsection :: nonreg}.

We prove the $\sigma$ admits an embedding into $\pi$ as follows:
\begin{enumerate}
	\item In the case $\pi=\Ind_{M_{2}}^{G}(\Omega_{M_{2},-1})$, it holds that $f_n(\lambda_0)= 1$.
	By the central character argument, see \cite[Lemma. 3.12]{E6}, it follows that
	\[
	\sigma \hookrightarrow \Ind_T^G\lambda_0.
	\]
	Since $\sigma$ appears in $\Ind_T^G\lambda_0$ with multiplicity one, it follows that $\sigma$ is a subrepresentation of $\pi$.
	
	\item
	The representation $\pi=\Ind_{M_{4}}^{G}(\Omega_{M_{4},0})$ is unitary and hence semi-simple.
\end{enumerate}

We point out that the triple $(\sigma,w,\lambda_{a.d.})$, with $w$ to be the shortest Weyl word such that $w\cdot\lambda_0=\lambda_{a.d.}$, satisfies the assumptions given above.
However, in order to simplify the calculation of $\ker N_{w}(\lambda_0)\res{\pi}$
 as much as possible, it is preferable to choose $w$ and $\lambda_1$ so that $w$ will be as short as possible (in terms of the length function on $\weyl{G}$).
To that end, we consider all $w\in W^{M,T}$ such that
\[
f_n(w\cdot \lambda_0) = f_\pi(w\cdot \lambda_0) \neq 0
\]
and choose a $w$ of minimal length together with $\lambda_1=w\cdot\lambda_0$.
In this case, it holds that
\[
\inner{\lambda_0,\check{\gamma}}>0 \qquad \forall \gamma \in R(w)
\]
and hence, the operator $N_w(\lambda)$ is holomorphic at $\lambda_0$.
In particular, $N_{w}(\lambda_0)\res{\pi}$ is non-zero by the Gindikin-Karpelevich formula (see \cite[Chapter 4]{MR0419366}).

In the following table, we list the relevant data ($\lambda_{a.d.}$, $\lambda_0$, $\lambda_1$ and $w$) for each case.
\begin{longtable}[H]{|c|c|c|}
	\hline
	& $\Ind_{M_2}^{G}(\Omega_{M_{2},-1})$  & $\Ind_{M_4}^{G}(\Omega_{M_{4},0})$    \\ \hline
	$\lambda_{a.d}$ & $\esevenchar{-1}{0}{0}{-1}{0}{0}{-1}$
	&
	$\esevenchar{0}{0}{0}{ 0}{-1}{0}{0}$ \\  
	\hline
	$\lambda_0$ & 
	$\esevenchar{-1}{5}{-1}{-1}{-1}{-1}{-1}$ &
	$\esevenchar{-1}{-1}{-1}{ 3}{-1}{-1}{-1}$
	\\
	\hline
	$\lambda_1$ & 
	$\esevenchar{-1}{-1}{3}{-1}{-1}{-1}{-1}$ &
	$\esevenchar{-1}{-1}{-1}{2}{-1}{-1}{-1}$ 
	\\
	\hline
	$w$
	&
	$\s{\alpha_7}\s{\alpha_6}\s{\alpha_5}\s{\alpha_4}\s{\alpha_2}$&
	$\s{\alpha_7}\s{\alpha_6}\s{\alpha_5}\s{\alpha_4}\s{\alpha_3}\s{\alpha_4}\s{\alpha_4}$ 
	\\ \hline 
	\caption{Proof \Cref{Prop::E7::iwahori} }
	\label{tab::Iwahori}
\end{longtable}
      
It remains to determine whether $N_{w}(\lambda_0)\res{\pi}$ is injective or not.
This is done in \Cref{section:: iwahori hecke algebra} with the following conclusions:
\begin{enumerate}
	\item In case (1) it is shown that  $N_{w}(\lambda_0)\res{\pi}$ is not injective and hence $\pi=\Ind_{M_{2}}^{G}(\Omega_{M_2,-1})$ admits a maximal semi-simple subrepresentation of length $2$.
	\item In case (2),  $N_{w}(\lambda_0)\res{\pi}$ is shown to be injective and hence $\pi=\Ind_{M_{4}}^{G}(\Omega_{M_4,0})$ admits a unique irreducible subrepresentation.
	Hence, it is irreducible.
\end{enumerate}

\end{proof}


\begin{appendices}
\appendix
\appendixpage

\section{Degenerate Principal Series of Levi Subgroups }\label{App:knowndata}
In this section, we list various branching rules used in this article.
We start by explaining how new branching rules can be inferred using the Aubert involution. We then make a list of various branching rules associated with Levi subgroups of $G$, organized by the type of the Levi subgroup. Most of these rules arise from irreducible degenerate principal series of the Levi subgroup, while some are proven by other methods.

Note that it is convenient to encode the branching rules in terms of the action of Weyl elements.
\subsection{Generalized Degenerate Principal Series}

Let $M$ be a Levi subgroup of a maximal parabolic subgroup of $H$. Let $\Omega$ be as in \Cref{formula:: character_form}. Let $\Pi_{triv,\Omega} = \Ind_{M}^{H}(\Omega)$. We set $\Pi_{St,\Omega}= \Ind_{M}^{H}(St_{M}\otimes \Omega)$, 
where 
$St_{M}$ is the Steinberg representation of $M$,
 to be the generalized degenerate principal series associated with $\Pi_{triv,\Omega}$, i.e. $\Pi_{St}$ is the Aubert involution of $\Pi_{triv,\Omega}$. By \cite[Lemma 3.1]{MR1951440}, the representation $\Pi_{triv,\Omega}$ is irreducible if and only if $\Pi_{St,\Omega}$ is irreducible. 
 
 Suppose that $\Pi_{triv,\Omega}$ is irreducible. Let $\lambda_{0}=\jac{M}{T}{\Omega}$ be the leading exponent of $\Pi_{triv,\Omega}$ and let $\lambda_{a.d}$ be an anti-dominant exponent of $\Pi_{triv}$. Since $\Pi_{triv,\Omega}$ is irreducible, 
 \[\mult{\lambda_{a.d}}{\jac{H}{T}{\Pi_{triv,\Omega}}}=|Stab_{\weyl{H}}(\lambda)|.\]
 Thus, there exists a unique irreducible representation $\sigma$ of $H$ having $\mult{\lambda_{a.d}}{\jac{H}{T}{\sigma}} \neq 0$, namely,  $\sigma=\Pi_{triv,\Omega}$. Fix $u \in \weyl{H}$ such that $u \lambda_0 =\lambda_{a.d}$.  \Cref{Lemma::geomtric_lema} implies that 
 \[\jac{G}{T}{\Pi_{triv,\Omega}} = \sum_{w\in W^{M,T}} \coset{w\lambda_0}= 
 \sum_{w \in W^{M,T}} \coset{wu^{-1}\lambda_{a.d}}.\]

Since $\Pi_{triv,\Omega}$ is irreducible, so is $\Pi_{St,\Omega}$. Let $\lambda_1 = \jac{M}{T}{St \otimes \Omega} = \Omega\res{T} + \rho_{M}$, and let $\lambda_{d} $ be a dominant exponent in the orbit $\weyl{H} \cdot \lambda_1$, i.e.
for every $\alpha \in \Delta_{H}$ one has $\Real(\inner{\lambda_{d},\check{\alpha}}) \geq 0$.
Using \cite[Subsection 5.3.2]{Lie}, it follows that $\lambda_{d} \leq \jac{H}{T}{\Pi_{St,\Omega}}$. 
Fix $u_{d} \in \weyl{H}$, such that $u_{d}\lambda_1 =\lambda_{d}$ 
then 
\[\jac{H}{T}{\Pi_{St,\Omega}}=\sum_{w \in W^{M,T}} \coset{w \lambda_1} =
\sum_{w \in W^{M,T}} \coset{w u_{d}^{-1}\lambda_d}.\]

\subsection{Different types of Rules}
\subsubsection{Orthogonality Rule}
We recall the Orthogonality Rule from \cite[A.1]{E6}. Let $\lambda \in \bfX(T)$ and set $\Theta_{\lambda}= \set{\alpha \: : \:  \alpha \in \Delta_{G}, \; \inner{\lambda,\check{\alpha}}=0}$. Then 
\begin{framed}
\begin{equation}\tag{OR}
\label{Eq:OR}
\lambda \leq\jac{H}{T}{\pi} \Longrightarrow \Card{W_{M_{\Theta_\lambda}}} \times \coset{\lambda} \leq \coset{\jac{H}{T}{\pi}} .
\end{equation}
\end{framed}
A direct consequence of \eqref{Eq:OR} is the following Lemma.
\begin{Lem}\label{Lemma::unique anti}
Let $M$ be a maximal  Levi subgroup of $G$, $\Omega \in \bfX(M)$ such that $\Omega=\Real(\Omega)$ and let $\pi =\Ind_{M}^{G}(\Omega)$. Then: 
\begin{enumerate}[ref=\Cref{Lemma::unique anti}.(\arabic*)]
\item\label{Lemma::unique anti::1}
$\jac{H}{T}{\pi}$ contains a unique anti-dominant exponent $\lambda_{a.d}$.
\item \label{Lemma::unique anti::2}
Let $\sigma$ be an irreducible constituent of $\pi$, having $\mult{\lambda_{a.d}}{\jac{H}{T}{\sigma}}\neq 0$. Then
$\mult{\lambda_{a.d}}{\jac{H}{T}{\pi}} = \mult{\lambda_{a.d}}{\jac{H}{T}{\sigma}} =|\Stab_{\weyl{H}}(\lambda_{a.d})|$.
\item \label{Lemma::unique anti::3}
$\sigma$ is the unique irreducible representation of $H$ with the property $\lambda_{a.d} \leq \jac{H}{T}{\sigma}$.
\end{enumerate}   
\end{Lem}
\begin{proof}
\begin{enumerate}
\item
The existence of $\lambda_{a.d}$ follows from \cite[Lemma 3.4]{E6}. The uniqueness is due to the assumption $\Real(\Omega)=\Omega$.
\item
Follows from \eqref{Eq:OR}.
\item
Follows from the second part.
\end{enumerate}
\end{proof}
 
\subsubsection{  Rules Coming From Levi Subgroups Of Type $A_n$}
We fix the following labeling of the Dynkin diagram of a group $H$ of type $A_n$. 
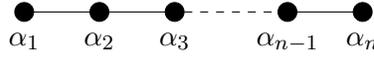
\[
\begin{tikzpicture}[scale=0.5]
\draw (-1,0) node[anchor=east]{};
\draw (0 cm,0) -- (4cm,0);
\draw[dashed] (7 cm, 0) -- (4cm,0);
\draw (7 cm,0) -- (9cm,0);
\draw[fill=black] (0 cm, 0 cm) circle (.25cm) node[below=4pt]{$\alpha_1$};
\draw[fill=black] (2 cm, 0 cm) circle (.25cm) node[below=4pt]{$\alpha_2$};
\draw[fill=black] (4 cm, 0 cm) circle (.25cm) node[below=4pt]{$\alpha_3$};
\draw[fill=black] (7 cm, 0 cm) circle (.25cm) node[below=4pt]{$\alpha_{n-1}$};
\draw[fill=black] (9 cm, 0 cm) circle (.25cm) node[below=4pt]{$\alpha_{n}$};
\end{tikzpicture}
\]

We recall the Branching rule of type $A_2$ \cite[A.3]{E6}
\begin{framed}
	\begin{equation}\tag{$A_2$}\label{Eq::A2}
	\lambda\leq \jac{H}{T}{\pi},\ \gen{\lambda,\check{\alpha}} = \pm 1,\  \gen{\lambda,\check{\beta}} = 0 \quad \Longrightarrow \quad 2\times\coset{\lambda}+\coset{s_\alpha\cdot\lambda} \leq \coset{\jac{H}{T}{\pi}} .
	\end{equation}
\end{framed}

As a consequence, one gets 
\begin{Cor}\label{A2:: two_reps}
Suppose that $H$ is a group of type $A_2$.
Let $\pi = \Ind_{T}^{H}(\lambda)$ where $\lambda \in \set{\pm \fun{1}}$.  Then $\pi$ is of length two. Set $\sigma_{A_2}^{1}$ to be the unique irreducible subquotient of $\pi$ having   $\mult{\lambda}{\jac{H}{T}{\pi}}\neq 0$ and let $\sigma_{A_2}^{2}$ denote the other one. Then, one has that,
\[
\begin{split}
\jac{H}{T}{\sigma_{A_2}^{1}} &= 2 \times \coset{\lambda} +  \coset{\s{\alpha_1}\lambda} \\
\nonumber
\jac{H}{T}{\sigma_{A_2}^{2}} &= 2 \times \coset{\s{\alpha_2}\s{\alpha_1}\lambda} +  \coset{\s{\alpha_1}\lambda}.
\end{split}\]

Moreover, there are exactly two irreducible representations $\sigma$ of $G$ having $\s{\alpha_1}\lambda \leq \jac{G}{T}{\sigma}$.
\end{Cor}

\begin{proof}
Suppose $\lambda =-\fun{1}$. Applying (\ref{Eq::A2}) with respect to Levi subgroup $M_{\alpha_1}$ implies that 
\[\jac{H}{T}{\sigma_{A_2}^{1}} \geq 2 \times\coset{\lambda} + \coset{\s{\alpha_1}\lambda}.\]
Recall that the rule \eqref{Eq::A2} comes from the fact that the representation $\Ind_{M_{\alpha_2}}^{H}(\Omega_{-1})$ is irreducible.  Thus $\pi$, is at least of length $2$. Let $\sigma_{A_2}^{2}$ be an irreducible subquotient of $\pi$ having $\mult{\s{\alpha_2}\s{\alpha_1}\lambda} {\jac{H}{T}{\sigma_{A_2}^{2}}}\neq 0$. Applying \eqref{Eq::A2} on $\s{\alpha_2}\s{\alpha_1}\lambda$ with respect to the Levi subgroup $M_2 =M_{\alpha_2}$ implies that 
\[\jac{H}{T}{\sigma_{A_2}^{2}} \geq 2 \times \coset{\s{\alpha_2}\s{\alpha_1}\lambda} +  \coset{\s{\alpha_1} \lambda}.\]
Thus, 
\[\jac{H}{T}{\pi} \geq \jac{G}{T}{\sigma_{A_2}^{1}} + \jac{H}{T}{\sigma_{A_2}^{2}}.\]
In particular $\dim \jac{H}{T}{\pi} \geq 6$ since $|W_{H}| =6 $ we deduce that  \[\coset{\jac{H}{T}{\pi}} = \coset{\jac{H}{T}{\sigma_{A_2}^{1}}} +\coset{ \jac{H}{T}{\sigma_{A_2}^{2}}}.\]
\end{proof}
\begin{Lem}\label{A_3::example}
Let $\pi$ be an irreducible representation of a group $H$ of type $A_3$ having 
$\lambda \leq \jac{H}{T}{\pi}$ where
$\inner{\lambda,\check{\alpha_1}} =\pm 1 \quad \text{and} \quad  
\inner{\lambda,\check{\alpha_2}} =
\inner{\lambda,\check{\alpha_3}} =0$. Then, 

\begin{framed}
	\begin{equation} \tag{A.3(a)}
	\lambda\leq \jac{H}{T}{\pi}
	\quad \Longrightarrow  \quad
	6\times\coset{\lambda}
	+4\times \coset{\s{\alpha_1}\cdot\lambda}+ 2\coset{\s{\alpha_2}\s{\alpha_1}\cdot\lambda} \leq \coset{\jac{G}{T}{\pi}}
	\end{equation}
\end{framed}
\end{Lem}

\begin{proof}
Let $M= M_{\set{\alpha_1,\alpha_2}}$ and write $\coset{\jac{H}{M}{\pi}} =\sum n_{\sigma} \times \coset{\sigma}$, where $\sigma$ are irreducible  non-equivalent  representations of $M$. 
\begin{itemize}
\item
Since $\lambda \leq \jac{H}{T}{\pi}$, applying (\ref{Eq:OR}) implies that $6 \times \coset{\lambda} \leq \jac{H}{T}{\pi}$.
\item
 On the other hand, since $|\Stab_{\weyl{H}}(\lambda)|=6$, it follows that 
\[\mult{\lambda}{\jac{H}{T}{\pi}}=6.\]
\item
By (\ref{Eq::A2}) it holds that, 
\[\lambda \leq \jac{H}{T}{\pi} \Rightarrow 2\times \coset{\lambda} + \coset{\s{\alpha_1}\lambda} \leq \coset{\jac{H}{T}{\pi}}\]
\item
Combing the above yields $\mult{\sigma_1^{A_2}}{\jac{H}{M}{\pi}}= 3$. In particular, $\mult{\s{\alpha_1}\lambda}{\jac{H}{T}{\pi}} \geq 3$.   
\item
Applying \eqref{Eq:OR} on $\s{\alpha_1}\lambda$ yields
\[|\weyl{{\Theta}}| \:  \divides \: \mult{\s{\alpha_1}\lambda}{\jac{H}{T}{\pi}}\]
where  
$\Theta = \set{\alpha \: : \: \alpha \in \Delta_{H} \: :  \:  \inner{\s{\alpha_1}\lambda,\check{\alpha}}=0}$.
Hence,  $\mult{\s{\alpha_1}\lambda}{\jac{H}{T}{\pi}}\geq 4$.
\item
Since, up to equivalence, the only irreducible representations $\tau$ of $M$ such that $\s{\alpha_1}\lambda \leq \jac{M}{T}{\tau}$ are $\sigma^{1}_{A_2}$ and $\sigma^{2}_{A_2}$ it follows that $n_{\sigma_{A_2}^{2}}\geq 1$. 
\item 
In summary 
\[\jac{H}{M}{\pi} \geq 3 \times \coset{\sigma_{A_2}^{1}} + 1 \times \coset{\sigma_{A_2}^{2}}.\]
\item
Thus,
\[
\begin{split}
\jac{H}{T}{\pi}  
\geq& 3 \times \coset{\jac{M}{T}{\sigma_{A_2}^{1}}} + 1 \times \coset{\jac{M}{T}{\sigma_{A_2}^{2}}} \\
=& 6 \times \coset{\lambda} + 3 \times \s{\alpha_1}\lambda + 2 \times\s{\alpha_2}\s{\alpha_1} \lambda + 1\s{\alpha_1}\lambda\\
=& 
6 \times \coset{\lambda} + 4 \times \s{\alpha_1}\lambda + 2 \times\s{\alpha_2}\s{\alpha_1} \lambda
\end{split}\]
\end{itemize}
\end{proof}
A similar argument yields the following, more general, rule

\begin{Lem}
Let $H$ be a group of type $A_n$ where $n\geq 2$ and let $\pi$ be an irreducible representation of $H$ having $\mult{\lambda}{\jac{H}{T}{\pi}} \neq0$, where $\lambda \in \set{\pm \fun{1}}$. Set $M =M_{\set{\alpha_1 ,\dots \alpha_{n-1}}}$. Then 
\begin{framed}
	\begin{equation}\tag{$A_n$}
	\lambda\leq \jac{H}{T}{\pi}
	\Longrightarrow
	\sum_{w \in W^{M,T}} 
	(n-l(w)) \cdot (n-1)! \coset{w \lambda}
	= \coset{\jac{H}{T}{\pi}}.
	\end{equation}
\end{framed}
\end{Lem}

We record all  Branching rules of type $A_n$ that we used in this paper. For more information, one should consult \cite[Appendix A]{E6}.
The labeling  of the rule indicates which type of Levi subgroup we refer to.
\begin{framed}
\begin{equation}
\tag{$A_1$}
\lambda\leq \jac{H}{T}{\pi},\ \gen{\lambda,\check{\alpha}} \neq \pm 1 \Longrightarrow \coset{\lambda}+\coset{s_\alpha\cdot\lambda} \leq \coset{\jac{H}{T}{\pi}} .
\end{equation}
\end{framed}

\begin{framed}
	\begin{equation}\tag{$A_2$}
	\lambda\leq \jac{H}{T}{\pi},\ \lambda \in \set{\pm \fun{1}} \Longrightarrow 2\times\coset{\lambda}+\coset{\s{\alpha_1}\cdot\lambda} \leq \coset{\jac{H}{T}{\pi}} .
	\end{equation}
\end{framed}
\begin{framed}
	\begin{equation}
	\tag{$A_3$}
	\begin{array}{c}
	\lambda\leq \jac{H}{T}{\pi},\ 
	\lambda \in \set{\pm (\fun{1} -\fun{3})} \\ 
	\Longrightarrow 2\times\coset{\lambda} + \coset{\s{\alpha_1}\cdot\lambda} + \coset{\s{\alpha_3}\cdot\lambda} + 2\times \coset{\s{\alpha_1} \s{\alpha_3}\cdot\lambda} \leq \coset{\jac{H}{T}{\pi}} .
	\end{array}
	\label{Eq::A3b}
	\end{equation}
\end{framed}
\begin{framed}
	\begin{equation}\tag{$A_n$}
	\lambda\leq \jac{H}{T}{\pi}, \, \lambda=\pm\fun{1}
	\Longrightarrow
	\sum_{w \in W^{M,T}} 
	(n-l(w)) \cdot (n-1)! \coset{w \lambda}
	= \coset{\jac{H}{T}{\pi}}
	\label{Eq:: genan}
	\end{equation}
where $M= M_{\set{\alpha_2 \dots \alpha_{n}}}$.
\end{framed}

\subsubsection{  Rules Coming From Levi Subgroups Of Type $D_n$}

We fix the following labeling of the Dynkin diagram of a group $H$ of type $D_n$. 
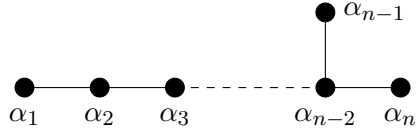
\[\begin{tikzpicture}[scale=0.5]
\draw (0 cm,0) -- (4 cm,0);
\draw[dashed] (4 cm,0) -- (8 cm,0);
\draw (8 cm,0) -- (10 cm,0 cm);
\draw (8 cm,0) -- (8 cm,2 cm);
\draw[fill=black] (0 cm, 0 cm) circle (.25cm) node[below=4pt]{$\alpha_1$};
\draw[fill=black] (2 cm, 0 cm) circle (.25cm) node[below=4pt]{$\alpha_2$};
\draw[fill=black] (4 cm, 0 cm) circle (.25cm) node[below=4pt]{$\alpha_3$};
\draw[fill=black] (8 cm, 0 cm) circle (.25cm) node[below=4pt]{$\alpha_{n-2}$};
\draw[fill=black] (10 cm, 0 cm) circle (.25cm) node[below=4pt]{$\alpha_n$};
\draw[fill=black] (8 cm, 2 cm) circle (.25cm) node[right=3pt]{$\alpha_{n-1}$};
\end{tikzpicture}\]

\begin{Lem}
Let $H$ be a group of type $D_5$.
Let $\pi$ be an irreducible representation of $H$ having $\mult{\lambda}{\jac{G}{T}{\pi}}\neq 0$, where $\lambda =\fun{5}$. Then, one has the following rule
\[
\begin{split}
\lambda \leq \jac{H}{T}{\pi} & \Rightarrow 120 \times  \coset{\lambda} + 
96\times \coset{\s{\alpha_5}\lambda} + 
72\times \coset{\s{\alpha_3}\s{\alpha_5}\lambda}\\
& + 48\times \coset{\s{\alpha_2}\s{\alpha_3}\s{\alpha_5}\lambda}
+ 48\times \coset{\s{\alpha_4}\s{\alpha_3}\s{\alpha_5}\lambda}
+ 32\times \coset{\s{\alpha_4}\s{\alpha_2}\s{\alpha_3}\s{\alpha_5}\lambda}
\\
& + 24\times \coset{\s{\alpha_1}\s{\alpha_2}\s{\alpha_3}\s{\alpha_5}\lambda}
+ 16\times \coset{\s{\alpha_3}\s{\alpha_2}\s{\alpha_4}\s{\alpha_3}\s{\alpha_5}\lambda}
+ 16\times \coset{\s{\alpha_1}\s{\alpha_2}\s{\alpha_4}\s{\alpha_3}\s{\alpha_5}\lambda}\\
&+ 8\times \coset{\s{\alpha_3}\s{\alpha_1}\s{\alpha_2}\s{\alpha_4}\s{\alpha_3}\s{\alpha_5}\lambda}.
\end{split}
\]

\end{Lem}
\begin{proof}
$ $
\begin{itemize}
\item
By \eqref{Eq:OR}, one has $120\times \lambda \leq \jac{G}{T}{\pi}$.
\item
Applying \eqref{Eq:: genan}, with respect to  $M_{4}=M_{\set{\alpha_1,\alpha_2,\alpha_3,\alpha_5}}$ yields 
$\mult{\sigma_{A_4}^{1}}{\jac{H}{M_4}{\pi}}=5$. In particular, one has 
\[120\times \coset{\lambda} +
 90 \times \coset{\s{\alpha_5}\lambda} +
 60 \times  \coset{\s{\alpha_3}\s{\alpha_5}\lambda} +
 30 \times \coset{\s{\alpha_2}\s{\alpha_3}\s{\alpha_5}\lambda} \leq \jac{H}{T}{\pi}.\]
\item
By \eqref{Eq:OR}, one has
\[12 \divides \mult{\s{\alpha_5}\lambda}{\jac{H}{T}{\pi}}.\]
Thus, $\mult{\s{\alpha_5}\lambda}{\jac{H}{T}{\pi}}\geq 96$. In particular, $\mult{\sigma_{A_4}^{2}}{\jac{H}{M_4}{\pi}}\geq 1$.
\item
We conclude that, 
\[\jac{H}{M_4}{\pi} \geq 5 \times \sigma_{A_4}^{1} + \sigma_{A_4}^{2}.\]
Thus, 
\[\begin{split}
\jac{H}{T}{\pi} \geq & 
120\times \coset{\lambda} +
 96 \times \coset{\s{\alpha_5}\lambda}  
 +72 \times  \coset{\s{\alpha_3}\s{\alpha_5}\lambda} 
 \\&+48 \times \coset{\s{\alpha_2}\s{\alpha_3}\s{\alpha_5}\lambda}
 +
 24 \times \coset{\s{\alpha_1}\s{\alpha_2}\s{\alpha_3}\s{\alpha_5}\lambda}.
\end{split}\]

\item
On the other hand, applying \eqref{Eq:: genan} on $\lambda$ with respect to $M_{1,2}= M_{\set{\alpha_3,\alpha_4,\alpha_5}}$ yields that 
$\mult{\sigma_{A_3}^{1}}{\jac{G}{M_{1,2}}{\pi}}=20$. 
\item
Since $\mult{\s{\alpha_5}\lambda}{\jac{H}{T}{\pi}}\geq 96$ and $20\times \sigma_{A_3}^{1}$ contributes only $80$ copies of $\s{\alpha_5}\lambda$, we deduce that 
\[\jac{H}{M_{1,2}}{\pi} \geq 20 \times \sigma_{A_3}^{1} + 8 \times  \sigma_{A_3}^{2}.\]
In particular, one has $48 \times \coset{\s{\alpha_4}\s{\alpha_3}\s{\alpha_5}\lambda} \leq \jac{H}{T}{\pi}$.
\item
Applying \eqref{Eq:: genan} on $\s{\alpha_2}\s{\alpha_3}\s{\alpha_5}\lambda$ yields that 
\[32\times \coset{\s{\alpha_4}\s{\alpha_2}\s{\alpha_3}\s{\alpha_5}\lambda} 
+
16\times \coset{\s{\alpha_3}\s{\alpha_4}\s{\alpha_2}\s{\alpha_3}\s{\alpha_5}\lambda} 
\leq \jac{H}{T}{\pi}.\]
\item
By Applying \eqref{Eq::A3b} on $16\times \coset{\s{\alpha_3}\s{\alpha_4}\s{\alpha_2}\s{\alpha_3}\s{\alpha_5}\lambda}$ with respect to $M_{\alpha_1,\alpha_2,\alpha_3}$ 
one has
\[16\times \coset{\s{\alpha_1}\s{\alpha_4}\s{\alpha_2}\s{\alpha_3}\s{\alpha_5}\lambda}
+
8\times \coset{\s{\alpha_1}\s{\alpha_3}\s{\alpha_4}\s{\alpha_2}\s{\alpha_3}\s{\alpha_5}\lambda} \leq \jac{H}{T}{\pi}.\]
\end{itemize}
This completes the proof.
\end{proof}
\begin{Lem}\label{Lemma ::result D6}
Let $H$ be a group of type $D_6$.
There exists a unique irreducible subquotient $\sigma_1$ of $H$  having $\mult{\lambda_{a.d}}{\jac{H}{T}{\sigma_1}}\neq 0$, where 
$\lambda_{a.d} = \dsixcharchar{-1}{0}{0}{-1}{0}{0}$. In that case, one has 
$\mult{\lambda_{a.d}}{\jac{G}{T}{\sigma_1}}=24$ and $\mult{\lambda_0}{\jac{H}{T}{\sigma_1}} =1$, where 
\[\lambda_0= \s{\alpha_3}\s{\alpha_2}\s{\alpha_1}\s{\alpha_4}
\s{\alpha_3}\s{\alpha_5} \s{\alpha_6}\s{\alpha_4}\cdot \lambda_{a.d}.\]
 
\end{Lem}
\begin{proof}
Let $\pi =\Ind_{M_3}^{G}(\Omega_{0})$. By \cite[Theorem 5.3]{MR2017065}, $\pi$  is semi-simple of length two. We write $\pi =\sigma_1\oplus \sigma_2$. Applying \Cref{Lemma::unique anti}, there exists a unique irreducible subrepresentation $\tau$ of $\pi$ having $\mult{\lambda_{a.d}}{\jac{G}{T}{\tau}}\neq 0$, say $\tau =\sigma_1$ and $\mult{\lambda_{a.d}}{\jac{H}{T}{\sigma_1}}= 24$. 

 Note that $\lambda_0=\jac{M_{3}}{T}{\Omega_{0}}$.  By \ref{Prop::unique::1} , $\mult{\lambda_0}{\jac{G}{T}{\sigma_i}}\neq 0$, while, by \Cref{Lemma::geomtric_lema}, one has $\mult{\lambda_0}{\jac{G}{T}{\pi}}=2$. Thus the claim follows. 
\end{proof}

\begin{Lem}\label{D5:reseults}
Let $H$ be a group of type $D_5$ and let  $\pi = \Ind_{M_3}^{H}(\Omega_0)$. Then,
$\pi$ is irreducible and $\M_{u}\res{\pi}$,
where $u = \s{\alpha_3}\s{\alpha_5}\s{\alpha_4}\s{\alpha_2}\s{\alpha_3}$,
 is injective. 

\end{Lem}
\begin{proof}
By \cite[Theorem 5.3]{MR2017065}, $\pi$ is irreducible. Let $\lambda_0 =\jac{M_3}{T}{\Omega_0}$ be the leading exponent of $\pi$ and let $\lambda_{a.d}$ be the anti-dominant exponent of $\pi$. Note that $u\lambda_0 =\lambda_{a.d}$. Thus, it  induces a map between 
$\Ind_{M_3}^{H}(\Omega_0)$ and $\Ind_{T}^{H}(\lambda_{a.d})$.
Since $\pi$ is irreducible, and $\M_{u}\res{\pi}\neq 0$ the map is an isomorphism.
\end{proof}

\subsubsection{ Rules Coming From Levi Subgroups Of Type $E_6$}
Let $H$ be a group of type $E_6$.
We fix the following labeling of the Dynkin diagram of group of type $E_6$. 

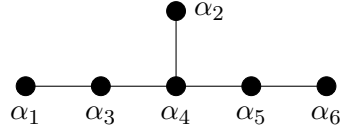
\[\begin{tikzpicture}[scale=0.5]
\draw (-1,0) node[anchor=east]{};
\draw (0 cm,0) -- (8 cm,0);
\draw (4 cm, 0 cm) -- +(0,2 cm);
\draw[fill=black] (0 cm, 0 cm) circle (.25cm) node[below=4pt]{$\alpha_1$};
\draw[fill=black] (2 cm, 0 cm) circle (.25cm) node[below=4pt]{$\alpha_3$};
\draw[fill=black] (4 cm, 0 cm) circle (.25cm) node[below=4pt]{$\alpha_4$};
\draw[fill=black] (6 cm, 0 cm) circle (.25cm) node[below=4pt]{$\alpha_5$};
\draw[fill=black] (8 cm, 0 cm) circle (.25cm) node[below=4pt]{$\alpha_6$};
\draw[fill=black] (4 cm, 2 cm) circle (.25cm) node[right=3pt]{$\alpha_2$};
\end{tikzpicture}\]
\begin{Lem}\label{E6:reseults}

\begin{enumerate}
	\item The representation $\pi =\Ind_{ M_5}^{H}(\Omega_{-\frac{1}{2}})$ is irreducible. \item $\M_{u}\res{\pi}$, where $u=\s{\alpha_4}.
\s{\alpha_5}\s{\alpha_6}\s{\alpha_3}\s{\alpha_2}\s{\alpha_4}\s{\alpha_5}$, is injective.
\end{enumerate}
\end{Lem}
\begin{proof}
The first part follows from \cite{E6}. For the second part, we argue as in \Cref{D5:reseults}. 
\end{proof}
\section{The Iwahori-Hecke Algebra and The Unramified Principal Series} \label{section:: iwahori hecke algebra}
 
In this section we recall the theory of finitely-generated modules over the Iwahori-Hecke algebra of $G$ and their relation to the unramified principal series of $G$ and use it to complete the proof of \Cref{Prop::E7::iwahori}.
For more information on the structure of the Iwahori-Hecke algebra and its modules see \cite{MR2250034}  and \cite{MR2642451}.

\subsection{Notations}

As before, $F$ is a non-Archimedean  local field. Let $\mathcal{O}$ denote its ring of integers and $\mathcal{P}$ be the maximal ideal of $\mathcal{O}$. Let $q =|\mathcal{O}\slash \mathcal{P}|$ and $\F_{q}\simeq \mathcal{O}\slash \mathcal{P}$, the field of $q$ elements. Let $\G$ be a split, semi-simple, simply-connected group scheme  such that $G=\G(F)$ and assume that $\mathbb{G}$ is defined over $\mathcal{O}$.
Let $\mathbb{B},\mathbb{T}$ be a Borel subgroup and a maximal split torus such that $\mathbb{B}(F)= \para{B}$ and $\mathbb{T}(F)= \para{T}$. 
We fix a hyper-special maximal compact subgroup $K =\G(\mathcal{O})$ of $G$ and let 
\[\Psi\: : \: K \: \rightarrow \G(\F_{q})\]
denote the projection modulo $\mathcal{P}$.

We note that $\bfX^{un}(T)$ is the group of all characters $\lambda\in\bfX(T)$ such that $\lambda\res{T\cap K}$ is trivial.
It is possible to extend the usual pairing between rational characters and co-characters to $\bfX^{un}(T)$ (see \Cref{subsec :: intertwining_operators}) by
\[
\lambda(\check{\alpha}(\unif)) = q^{\inner{\lambda,\check{\alpha}}} ,
\]
where $\unif$ is a generator of $\mathcal{P}$.
 
Let $J = \Psi^{-1}(\mathbb{B}(\F_{q}))$ be an Iwahori subgroup of $G$. The subgroup $J$ plays an important role in the study of  unramified principal series representations of $G$. By \cite[Proposition 2.7]{MR571057}, 
if $\lambda$ is unramified, then $\Pi=\Ind_T^G(\lambda)$ is generated by its Iwahori fixed vectors and so are all of its subquotients. 

We continue by recalling the definition of the Iwahori-Hecke algebra $\Hecke$. The algebra $\Hecke$ consists of all compactly supported $J$ -bi-invariant complex functions on $G(F)$, namely,
\[\Hecke = \set{ f \in C_{c}(G(F)) \: : \: f(j_1 g j_2) =f(g)  \quad \forall j_1,j_2 \in J, \quad g \in G(F)}.\]
The multiplication in $\Hecke$ is given by convolution and the measure of $J$ is set to be $1$. By \cite[Section 3]{Casselman} and \cite{Borel1976},
there is an equivalence of
categories between
the category of
admissible representations of $G$ which are generated by their $J$-fixed vectors
and the category of finitely generated $\Hecke$-modules.
This equivalence of categories sends an admissible representation $\pi$ of $G$ to the $\Hecke$-module $\pi^J$ of $J$-fixed vectors in $\pi$.
Thus, in order to study the structure of $\Ind_{T}^{G}(\lambda)$ it is sufficient to study the  corresponding finite dimensional $\Hecke$-module.  

\subsection{The Bernstein Presentation And Unramified Principal Series}
The Iwahori-Hecke algebra, $\Hecke$, can be described in terms of generators and relations. One such presentation is known as the Bernstein's presentation.
$\Hecke$ is generated by a set of generators $\set{T_{\s{\alpha}},\theta_{\check{\alpha}} \: : \: \alpha \in \Delta_{G}}$ subject to certain relations listed in \cite[Section 3]{MR2250034}. The algebra $\Hecke$ admits two important subalgebras: 
\begin{itemize}
\item
A finite dimensional algebra $\Hecke_{0} = \gen{ T_{\s{\alpha}} \: \: : \: \alpha \in \Delta_{G}} =\Span_{\C} \set{T_{\w} \: : \: \w \in \weyl{G}}$  of dimension $|\weyl{G}|$. 
\item
An infinite dimensional commutative algebra 
\[\Theta = \gen{\theta_{\check{\alpha}} \: :\:  \alpha \in \Delta_{G}} = \Span_{\C}\set{\theta_{x} \: : \: x \in \Z[\check{\Delta_{G}}] },\]
where $\Z[\check{\Delta_{G}}]$ is the co-root lattice of $T$.
\end{itemize}
In particular, as vector spaces,
\[
\Hecke = \Hecke_{0} \otimes \Theta .
\]

Given an unramified principal series $\Pi = \Ind_{T}^{G}(\lambda)$, we describe the left $\Hecke$-module, $\Pi^J=\Hecke(\lambda)$, corresponding to it by the equivalence of categories of \cite{Borel1976} using the Bernstein presentation.
This module is given by the left $\Hecke$-action on $\Hecke(\lambda)=\Hecke\otimes_\Theta \C_\lambda$, where $\C_\lambda$ is the one-dimensional representation of $\Theta$, given by $\lambda$.
In other words, $\Hecke(\lambda)$ can be identified, as a vector space, with $\Hecke_0$, while the action of $\Hecke$ is given as follows:
\begin{itemize}
\item
The action of $\Hecke_{0} \leq \Hecke$ on $\Hecke(\lambda)$ is given by left multiplication.
\item
By the Bernstein presentation, the action of $\Theta$ on $\Hecke(\lambda)$ is determined by the action of the generators $\theta_{\check{\alpha}}\in \Theta$ on $T_e\in \Hecke(\lambda)$. Let 
\[
\theta_{\check{\alpha}} \cdot T_e = q^{\inner{\lambda,\check{\fun{\alpha}}}}T_{e}.
\]
\end{itemize}
 

\subsection{Intertwining Operators} \label{subsection::itertwining}

We recall the normalized intertwining operators $N_{w}(\lambda)$, which were introduced in \Cref{subsec :: intertwining_operators} for $\lambda\in\bfX^{un}(T)$.
For a subrepresentation $\pi$ of $\Pi=\Ind_T^G(\lambda)$, these operators induce a map $N_{w}(\lambda)\res{\pi^J}$ of $\Hecke$-modules.
By \cite[Section 2]{MR2642451}, the action of $N_{\s{\alpha}}(\lambda)\res{\Pi^J}$ is given by \textbf{right}-multiplication by the following element
\[n_{s_\alpha}(\lambda)  = \frac{q-1}{q^{z+1}-1} T_{e} + \frac{q^{z}-1}{q^{z+1}-1}T_{\s{\alpha}} \in \Hecke_{0},\]
where $z=  \inner{\lambda,\check{\alpha}}$.

Suppose that $\Real(z)>-1$.
Then, $N_{\s{\alpha}}(\lambda)$ is holomorphic there.
Furthermore, considered as an element of $\operatorname{End}(\Hecke_0)$, $N_{s_\alpha}(\lambda)\res{\Pi^J}$ is a diagonalizable linear operator with two eigenvalues given by
\[
\lambda_1 =\frac{q-1}{q^{z+1} -1} + q\frac{q^{z}-1}{q^{z+1} -1}=1,\quad 
\lambda_2 =\frac{q-1}{q^{z+1} -1} - \frac{q^{z}-1}{q^{z+1} -1} = \frac{q-q^z}{q^{z+1} -1},
\]
with the exception of $z=0$, where $n_{s_\alpha}(\lambda)=T_e$ is the identity element and $N_{s_\alpha}(\lambda)=\operatorname{Id}$.

Thus, 
$N_{s_\alpha}(\lambda)\res{\Pi^J}$ has a kernel if and only if $\lambda_2=0$, which happens only if $z\in 1+\frac{2 \pi i}{log(q)} \Z$.
It follows that, for $z\in\R$, the injectivity of $N_{s_\alpha}(\lambda)\res{\Pi^J}$ does not depend on the value of $q$.

\subsection{The Submodule $\Hecke_{\para{P}}(\Omega)$}\label{subsection:: hPmoduule}

Let $\para{P}$ be a parabolic subgroup of $G$ with Levi part $M$ and let $\chi$ be an unramified character of $M$ with respect to $M \cap K$.
We denote the longest Weyl element of $\weyl{M}$ by $w_{M}^{0}$.
Let $\pi=\Ind_M^G\Omega$, with $\Omega\in\bfX(M)$ unramified, $\lambda_0=\jac{M}{T}{\Omega}$ and
\[
\Hecke_{\para{P}}(\Omega) = \pi^J .
\]

We recall that
\[
\pi=\Ind_{M}^{G}(\Omega) =  \Image \M_{w_{M}^{0}}(\Omega+\rho_{M})= \Image N_{w_{M}^{0}}(\Omega+\rho_{M}).
\]
It follows that $\Hecke_{\para{P}}(\Omega)$ is the image of $N_{w_{M}^{0}}(\Omega+\rho_{M})$ and hence, it has a basis given by
\begin{equation}
\label{eq::hPmodule_basis}
\set{ T_{u} \cdot triv \: : \:  u  \in \weyl{G}\slash \weyl{M}}, \quad \text{where} \quad triv= \sum_{w \in \weyl{M}}T_{w}.
\end{equation}


Under the equivalence of categories of \cite{Borel1976}, for any $\w \in \weyl{G}$, the operator $N_{w}(\lambda_0)\res{\pi}$ has a non-trivial kernel if and only if $N_{w}(\lambda_0)\res{\pi^{J}}$ does.
This, in turn, can be determined by calculating the rank of the matrix $\Lambda$
of $N_{w}(\lambda_0)\res{\pi^{J}}$ with respect to the basis of $\Hecke_{\para{P}}(\Omega)$ given in \Cref{eq::hPmodule_basis} and the basis of $\Hecke\bk{w\cdot\lambda_0}$ given by $\set{T_{\w} \: : \: \w \in \weyl{G}}$.

On the other hand, by the equivalence of categories described in \cite[Section 4]{MR2869018}, it holds that for any $w\in W^{M,T}$, the injectivity of $N_w\bk{\lambda_0}\res{\pi^J}$ does not depend on the value of $q$.
Thus, in order to determine whether $N_{w}(\lambda_0)\res{\pi^J}$ is injective for any value of $q$, it is enough to check it for a particular prime power $q$.
In the realization of the calculation described in \Cref{App:: Complixity}, we used the value $q=2$.

\subsection{Computing the Dimension of Kernels}\label{App:: Complixity}
We conclude this section by outlining the calculation of  $\dim_{\C} \bk{\Ker N_{w}(\lambda_0)\res{\pi^{J}}}$ required in the proof of \Cref{Prop::E7::iwahori}.
In preforming this calculation we had three limitations: computational speed, available \textit{RAM} (Random Access Memory) and hard drive space. 

In order to minimize computational time, all steps in the calculation were broken down to smaller steps which were calculated in parallel on a number of processors. This, in turn, resulted in a higher  \textit{RAM} usage.
Naively, finding the rank of the operator requires holding a matrix with $m$ columns , where $m =2,903,040$. However, such a matrix 
requires more  \textit{RAM}  than was available to us. 
We now explain how the calculation of the rank was organized so as to be completed in a reasonable amount of time with the resources available to us.

Fix a maximal parabolic subgroup $P$ with a Levi subgroup $M$. Let $\Omega \in \bfX(M)$ such that $\Omega =\Real(\Omega)$. In other words, $\Omega$ is unramified. Let $\pi =\Ind_{M}^{G}(\Omega)$ and $\lambda_0=\jac{M}{T}{\Omega}$.

Given $w \in \weyl{G}$, our goal is to determine whether the normalized intertwining operator, $N_{w}(\lambda_0)\res{\pi}$ has a kernel. This is equivalent to determine the codimension of the row space of $\Lambda$.
Since $\Hecke_0$ is of finite dimension, this is a problem in finite-dimensional linear algebra.




It is convenient to calculate the matrix $\Lambda$ using the element $n_{w}(\lambda_0)$.
The rows of the matrix are given by $v_{u} =  T_{u} \cdot triv \cdot n_{w}(\lambda_0)$ for $u \in W^{M,T}$.
In particular,
\[
\operatorname{rank}(\Lambda) = \dim \Span_{\C} \set{v_{u} \mvert u \in W^{M,T}}.
\]
Due to RAM limitations, this cannot be done in a straight-forward way and needs to be done in parts.
In order to generate the rows of $\Lambda$ we start by separately calculating
\[
v_{u,u',w}= T_{u}\cdot T_{u'} \cdot n_{w} \quad \forall u \in W^{M,T}, u' \in \weyl{M}
\]
and saving each one to the hard-drive.
It is then possible sum the elements
\[
v_{u} = \sum_{u' \in \weyl{M}} v_{u,u',w} \forall u \in W^{M,T} 
\]
and write each to the hard-drive.

While it is possible to calculate the coordinate vector of $v_{u}$ for each $u \in W^{M,T}$ separately, we were not able to load all of them at once and generate $\Lambda$, again due to RAM limitations.
However, we were able, by writing the coordinates into text files, to write the transposed matrix $\Lambda^T$ into a text file.
While $\operatorname{rank}(\Lambda)=\operatorname{rank}(\Lambda^T)$, it is simpler to compute the latter.
The idea is that $\Lambda^T$ is a matrix of dimension $|\weyl{G}| \times |W^{M,T}|$, instead of $|W^{M,T}| \times |\weyl{G}|$ (note the values of $|W^{M,T}|$ and $|\weyl{G}|$ given in \Cref{section::structure}).
It is then possible, to break $\Lambda^T$ into smaller blocks and perform the Gauss elimination process on each separately, then to combine the resulting non-zero rows to a new matrix and repeat the process until we are left with one matrix whose rows are linearly independent.
The rank of the resulting matrix equals $\dim_{\C} \bk{\Image N_{w}(\lambda_0)\res{\pi^{J}}}$, the co-dimension of the kernel.

We close by collecting the relevant data for the proof of \Cref{Prop::E7::iwahori}:
\begin{longtable}[H]{|c|c|c|}
	\hline
	& $\Ind_{M_2}^{G}(\Omega_{M_{2},-1})$  & $\Ind_{M_4}^{G}(\Omega_{M_{4},0})$    \\ \hline
	$w$
	&
	$\s{\alpha_7}\s{\alpha_6}\s{\alpha_5}\s{\alpha_4}\s{\alpha_2}$&
	$\s{\alpha_7}\s{\alpha_6}\s{\alpha_5}\s{\alpha_4}\s{\alpha_3}\s{\alpha_2}\s{\alpha_4}$ 
	\\ \hline 
	$\dim \Hecke_{\para{P}}(\Omega)$ & 576 & 10,080 \\
	\hline
	$\dim_{\C} \bk{\Image N_{w}(\lambda_0)\res{\pi^{J}}}$ & 561 & 10,080 \\ \hline
	$\dim_{\C} \bk{\Ker N_{w}(\lambda_0)\res{\pi^{J}}}$ & 15 & 0 \\ \hline
	\caption{Dimensions of kernels in \Cref{Prop::E7::iwahori} }
	\label{tab::Iwahori}
\end{longtable}

In particular, $N_{w}(\lambda_0)\res{\pi^{J}}$ is not injective in the case of $\pi=\Ind_{M_2}^{G}(\Omega_{M_{2},-1})$ and is injective in the case of $\pi=\Ind_{M_4}^{G}(\Omega_{M_{4},0})$.

\end{appendices}

\bibliographystyle{plain}
\bibliography{bib}
\end{document}